\documentclass[12pt, reqno]{amsbook}
\usepackage{hyperref}

\usepackage{amssymb,amsmath,graphicx,amsthm}

\def\contrazione{\raisebox{1pt}{\,{\mbox{\tiny{$|\!\raisebox{-0.7pt}{\underline{\hphantom{X}}}$}}}\,}}
\def\sumKk-1{\underset{|K|=k-1}{{\sum}'}}
\def\sumKq-1{\underset{|K|=q-1}{{\sum}'}}

\def\sumJ(p-1){\underset{|J|=p-1}{{\sum}'}}

\def\sumKq{\underset{|K|=q}{{\sum}'}}
\def\sumKp-2{\underset{|K|=p-2}{{\sum}'}}

\def\sumjq+1{\underset {j\leq q+1}\sum}
\def\sumjn-1{\underset {j\leq n-1}\sum}

\def\simto{\overset\sim\to}

\def\bt{\begin{theorem}}
\def\el{\end{lemma}}
\def\bl{\begin{lemma}}
\def\et{\end{theorem}}
\def\bp{\begin{proposition}}
\def\ep{\end{proposition}}
\def\bd{\begin{definition}}
\def\ed{\end{definition}}
\def\br{\begin{remark}}
\def\er{\end{remark}}

\def\simleq{\underset\sim<}

\def\simgeq{\underset\sim>}

\def\simle{\underset\sim<}

\def\simge{\underset\sim>}

\def\T{\text}

\def\1#1{\overline{#1}}

\def\2#1{\widetilde{#1}}

\def\3#1{\widehat{#1}}

\def\4#1{\mathbb{#1}}

\def\5#1{\frak{#1}}

\def\6#1{{\mathcal{#1}}}

\def\C{{\4C}}

\def\R{{\4R}}

\def\A{\6A}

\def\F{\6F}
\def\H{\6H}

\def\sumK{\underset{|K|=k-1}{{\sum}'}}

\def\sumJ{\underset{|J|=k}{{\sum}'}}

\def\sumij{\underset {ij\leq n-1}{{\sum}}}

\def\sumjq{\underset {j\leq q}\sum}

\def\Op{\T{Op}^{2\T{ord}(\Psi)-1}}
\def\Opp{\T{Op}^0}
\def\Opm{\T{Op}^{\T{ord}(\Psi)-\frac{1}{2}}}

\def\T{\text}
\newcommand{\Om}{\Omega}
\newcommand{\om}{\omega}

\newcommand{\bom}{\bar{\omega}}

\newcommand{\we}{\wedge}

\newcommand{\no}[1]{\|{#1}\|}

\newcommand{\No}[1]{\|{#1}\|^2}
\def\NO#1{||#1||^2}
\def\R{{\Bbb R}}

\def\C{{\Bbb C}}

\def\la{\langle}
\def\ra{\rangle}
\def\di{\partial}
\def\dib{\bar\partial}
\def\Label#1{\label{#1}}

\def\simto{\overset\sim\to}

\def\simleq{\underset\sim<}

\def\simgeq{\underset\sim>}

\def\simle{\underset\sim<}

\def\simge{\underset\sim>}

\def\T{\text}

\def\1#1{\overline{#1}}

\def\2#1{\widetilde{#1}}

\def\3#1{\widehat{#1}}

\def\4#1{\mathbb{#1}}

\def\5#1{\frak{#1}}

\def\6#1{{\mathcal{#1}}}

\def\C{{\4C}}

\def\R{{\4R}}

\def\A{\6A}

\def\sumK{\underset{|K|=k-1}{{\sum}'}}

\def\sumJ{\underset{|J|=k}{{\sum}'}}

\def\sumij{\underset {ij=1,\dots,N}{{\sum}}}

\def\sumjq{\underset{j=1}{\overset{q_0}\sum}}

\numberwithin{equation}{section}

\def\T{\text}

\theoremstyle{plain}

\newtheorem{theorem}{Theorem}[section]

\newtheorem{corollary}[theorem]{Corollary}

\newtheorem{lemma}[theorem]{Lemma}

\newtheorem{proposition}[theorem]{Proposition}

\theoremstyle{definition}

\newtheorem{definition}[theorem]{Definition}

\theoremstyle{remark}

\newtheorem{remark}[theorem]{Remark}

\author{Martino Fassina}
\title{Regularity of the $\bar{\partial}$-Neumann problem by means of superlogarithmic multipliers}

\begin{document}

\begin{center}
\vspace{10em}
{\LARGE\bfseries Regularity of the $\bar{\partial}$-Neumann problem\\[0.25em]
by means of superlogarithmic multipliers}

\vspace{2em}
{\Large Martino Fassina

\vspace{2em}
}
\end{center}

\cleardoublepage

\tableofcontents
\cleardoublepage
\chapter*{Abstract}
This thesis starts from a review on current research on the local hypoellipticity of the $\bar\partial$-Neumann problem. It presents the classical method of regularity from estimates of the energy: subelliptic as well as superlogarithmic. More recent material is included in which the regularity of the solution is obtained from the geometry of the singularities of the Levi form. The new contribution to this discussion consists in a general weighted Kohn-H\"ormander-Morrey formula twisted by a pseudodifferential operator. As an application, a new class of domains for which the $\dib$-Neumann problem is locally regular is exhibited.
\chapter*{Introduction}
The $\dib$-Neumann problem is probably the most important and natural example of a non-elliptic boundary value problem. The local hypoellipticity at the boundary of the canonical solution is usually established through the use of estimates such as subelliptic or superlogarithmic. In \cite{K00} Kohn showed that if the estimates fail but the points of failure are confined to a real curve transversal to the CR directions, local regularity for the tangential problem still holds (cf. \cite{K00} and \cite{BKZ14}). This is an exquisitely geometric conclusion. A fully geometric explanation of this phenomenon was given in \cite{BPZ14} by Baracco, Pinton and Zampieri.  They showed that good estimates in full are not needed and what really counts is that for a system of cut-off $\{\eta\}$ the gradient $\di\eta$ and the Levi form $\di\dib\eta$ are good multipliers in the sense of Kohn \cite{K79}. If these are subelliptic multipliers, then $\Box$ is hypoelliptic. The proof consists in modifying the Kohn-H\"ormander-Morrey formula by a weight $\phi=t|z|^2-\log\eta^2$; the exploitation  of $t|z|^2$ is usual in controlling the commutators $[\dib,\Lambda^s]$ and $[\dib^*,\Lambda^s]$, (where $\Lambda^s$ is the elliptic standard pseudodifferential operator of order s), but the one of $-\log\eta^2$ is new and is designed to overcome the error coming from the commutators $[\dib,\eta]$ and $[\dib^*,\eta]$. We want to  further generalize this result, using multipliers that are weaker than subelliptic. For this purpose we need a stronger modification of the Kohn-H\"ormander-Morrey formula in which not only the cut-off but also a general pseudodifferential operator appears as already commutated with $\dib$ and $\dib^*$. We thus present a general Kohn-H\"ormander-Morrey formula with weight in which a general pseudodifferential operator appears as a twisting term. Note that Kohn's work in \cite{K00} only deals with the tangential problem, but there is some literature for transferring regularity of $\Box_b$ to regularity of $\Box$ (cf. \cite{K02}, \cite{BKZ14}, \cite{KZ14}). This consists in exploiting the decomposition $Q=Q^\tau\oplus \bar L_n$, and requires the heavy technicalities of the harmonic extension. In order to avoid all this, the method of $\{\di\eta\}$-multipliers is very effective. Note that $\di\eta=(\di_b\eta,\di_\nu\eta)$. It is then sufficient that $\di_b\eta$ is a superlogarithmic multiplier: $\di_\nu\eta$ is $1$ but it hits $u_\nu$ which is 0 at $b\Om$ and therefore enjoys elliptic estimates. We thus have the following remark, simple but rich in consequences: if $\di_b\eta$ is a superlogarithmic multiplier for $\Box_b$, then it is also for the $\dib$-Neumann problem. 
 This fundamental observation, together with the twisted Kohn-H\"ormander-Morrey formula mentioned above, allow us to transfer to the $\dib$-Neumann problem the regularity result already established in \cite{BKPZ14} for the tangential problem.
 
\chapter{Preliminaries}
In this first chapter we give a brief description of the $\dib$-Neumann problem, for whose detailed account we refer to \cite{FK72}.
\section{The \texorpdfstring{$\dib$}{dbar}-Neumann problem}
Let $\Omega$ be a bounded domain of $\C^n$ with smooth boundary $b\Omega$, and $L_2^{(p,q)} (\Omega)$ the space of square integrable $(p,q)$-forms on $\Omega$, for $p,q \in \mathbb N$. The operator $\bar{\partial}$ defines a complex, namely the Cauchy-Riemann complex  \index{Notation !  $\dib^*$}\index{Notation ! $\dib$}
\begin{equation}\begin{split}
\underset{\dib^*}{\overset{\dib}{\rightleftarrows}}L_2^{p,q-1}(\Om)\underset{\dib^*}{\overset{\dib}{\rightleftarrows}} L_2^{p,q}(\Om)\underset{\dib^*}{\overset{\dib}{\rightleftarrows}} L_2^{p,q+1}(\Om)\underset{\dib^*}{\overset{\dib}{\rightleftarrows}}
 \end{split}\end{equation}
 where $\dib^*$ stands for the Hilbert-adjoint of $\dib$. We define the complex Laplacian $\Box:=\bar{\partial}\bar{\partial}^*+\bar{\partial}^*\bar{\partial}$ (if we need to specify degrees we use the notation
 $\Box_{(p,q)}$). We introduce the subsets $\H^{p,q}\subset L^{p,q}_2(\Om)$\index{Space ! $\H^{p,q}$} as 
\begin{equation}\begin{split}
\H^{p,q}=\{u\in \T{Dom}(\dib)\cap\T{Dom}(\dib^*)\big| \dib u=0 \T{ ~~and~~ } \dib^*u=0 \}.
 \end{split}\end{equation} 
Then, the $\bar{\partial}-Neumann\, problem$ can be stated as follows for $(p,q)$-forms. Given $\alpha\in L^{p,q}_2(\Om) $ with $\alpha \perp\H^{p,q}$, does there exist $u\in L^{p,q}(\Om)$, $u\perp \mathcal H^{p,q}$, such that
\begin{equation}\begin{split}
\Label{dN}
\begin{cases}\Box u=\alpha\\
u\in \T{Dom}(\bar{\partial})\cap \T{Dom}(\bar{\partial}^*)
\\
\bar{\partial} u\in \T{Dom}(\bar{\partial}^*), \bar{\partial}^* u\in \T{Dom}(\bar{\partial})?
\end{cases}
\end{split}\end{equation}
This question is intimately related to the so called $\bar{\partial}-problem$, which consists in finding, given $\alpha\in L_2^{p,q}$ such that $\bar{\partial}\alpha=0$, 
a form $u\in L_2^{p,q-1}$ such that $\bar\partial u=\alpha$.
The $\bar{\partial}$-problem can be solved for $\Omega$ pseudoconvex by using Hilbert spaces techniques and, moreover, the solution $u$ satisfies the estimate $\no u \lesssim \no f$ (where $\lesssim$ indicates inequality up to a multiplicative constant). From this we get the closed range of the complex Laplacian on pseudoconvex domains, so that it has a well defined inverse in $ \mathcal H^{\perp}$, namely the Neumann operator $N:=\Box^{-1}$. The following lemma will allow us to make further considerations.
 \begin{lemma}
For $1\leq q\leq n-1$ $Ker(\Box_{(p,q)})=\H^{p,q}=\{0\}$
\end{lemma}

\begin{proof}
It is obvious that $Ker(\Box_{(p,q)})\subseteq\H^{p,q}$. Let now $\alpha\in Ker\Box_{(p,q)}$. We have
$0=(\Box_{(p,q)}\alpha,\alpha)=(\bar{\partial}\bar{\partial}^*\alpha,\alpha)+(\bar{\partial}^*\bar{\partial}\alpha,\alpha)= (\bar{\partial}^*\alpha,\bar{\partial}^*\alpha)+(\bar{\partial}\alpha,\bar{\partial}\alpha)=
\no{\bar{\partial}^*\alpha}^2 + \no{\bar{\partial}\alpha}^2$. It follows $\bar{\partial}\alpha=\bar{\partial}^*\alpha=0$
hence $\alpha\in\H^{p,q}$. Now let us show that for $q\geq1$ we have $Ker(\Box_{(p,q)})={0}$. In fact, if $\alpha\in Ker(\bar{\partial})$ then, by what we have seen above, we have a solution $u\in L_2^{p,q-1}(\Om)$ to $\bar{\partial}u=\alpha$. Then $\no{\alpha}^2=(\bar{\partial}u,\alpha)=(u,\bar\partial^*\alpha)=0$.
\end{proof}
From this lemma we get immediately that the solution to the $\dib$-Neumann problem is unique and that $N$ is in fact an operator from $L_2^{p,q}$ to itself for $1\leq q\leq n-1$, so that in the $\bar{\partial}$-Neumann problem, for these values of $q$, the request that $u\in \mathcal H^{\perp} $ is in fact superfluous.
 
The Neumann operator thus defined provides in turn what is called the ''canonical solution'' to the $\bar{\partial}$-problem. 
 \begin{proposition}
$u=\bar{\partial}^*N\alpha$ is the unique solution to the $\bar{\partial}$-problem $\bar\partial u=\alpha$, $\bar\partial\alpha=0$, which is orthogonal to the Kernel, and it is called the canonical solution to the $\bar{\partial}$-problem.
\end{proposition}

\begin{proof}
Since $\bar{\partial}\Box N\alpha=0$ we have $\bar{\partial}\bar{\partial}^*\bar{\partial}N\alpha=0$. Therefore
$0=(\bar{\partial}\bar{\partial}^*\bar{\partial}N\alpha,\bar{\partial}N\alpha)=\no{\bar{\partial}^*\bar{\partial}N\alpha}^2$, and this implies $\bar{\partial}^*\bar{\partial}N\alpha=0$. Now, from $\Box N\alpha=\alpha$ we conclude that $\bar{\partial}\bar{\partial}^*N\alpha=\alpha$. Moreover $u\in \mathcal R(\bar{\partial}^*)\subset Ker(\bar{\partial})^\perp $.
\end{proof}
Similarly, the Neumann operator gives the canonical solution to the parallel $\bar\partial^*$-problem  $\bar\partial^*u=\alpha$, with $\bar\partial^*\alpha=0$. The canonical solution is the one orthogonal to $Ker(\bar\partial^*)$, and is given by $u=\bar\partial N \alpha$.\\
The $\dib$-Neumann problem is a non-elliptic boundary value problem; in fact,   the Laplacian  $\Box$ itself is elliptic but the boundary conditions which are imposed by the membership to $\T{Dom}(\Box)$,\index{Notation ! $\T{Dom}(\Box)$} that is, the second and third line of \eqref{dN}, are not.  The main interest relies in the {\it regularity} at the boundary for this problem, that is, in stating under which conditions $u$ inherits from $\alpha$ the smoothness at the boundary $b\Omega$ (it certainly does in the interior as we will see later in this chapter). We will focus on the local regularity of the Neumann operator, which is defined as follows.
\begin{definition}\Label{re}
\index{ Regularity ! local} The Neumann operator $N$ is regular at $z_0$ if  $\alpha\in C^\infty(U\cap  \bar\Om)$ implies $N\alpha\in C^\infty(U'\cap \bar\Om)$ where the sets $U'\subset U$ range in a system of  neighborhoods of a point $z_0\in \bar\Om$.
\end{definition}
As we will see in detail in the following chapter the main tools used in investigating the local regularity at the boundary of the $\dib$-Neumann problem  consist in certain a priori estimates such as subelliptic and superlogarithmic.

\section{Terminology and notations}
Let $\Om\subset \C^n$ be an open subset of $\C^n$ and let $b\Om$ denote the boundary of $\Om$. Throughout this thesis we  restrict ourselves to a domain $\Om$ with smooth boundary $b\Om$ defined by $r=0$ for $r$ a real-valued $C^\infty$ function such that $dr\ne 0$ in $b\Om$. Without loss of generality, we may assume that $r>0$ outside of $\bar\Om$ and $r<0$ in $\Om$.
 
For $z\in \C^n$, we denote by  $\C T_z\C^n$ the complexified tangent space to $\C^n$ at $z$. We have the direct sum decomposition $\C T_z\C^n=T^{1,0}_z\C^n\oplus T_z^{0,1}\C^n$, where $T^{1,0}_z\C^n$ and $T_z^{0,1}\C^n$ denote the holomorphic and anti-holomorphic tangent vectors at $z$ respectively. We will use for these spaces, in a neighborhood of a boundary point, an adapted frame of vector fields $L_1,\cdots, L_n,\bar L_1,\cdots,\bar L_n$, constructed as follows.

For $z_0\in b\Om$,  we fix $r$  so that $|\di r|=1$ in a neighborhood $U$ of $z_0$. We choose  $\om_1,\cdots, \om_n$ to be (1,0)-forms on $U$ such that $\om_n=\di r$ and such that $\la \om_i,\om_j\ra=\delta_{ij}$ for $z\in U$, where $\delta_{ij}$ is the Kronecker symbol. Let $L_1,\cdots, L_n$ be the dual basis of $\om_1 ,\cdots, \om_n$ on $U$, that is, for each $z\in U$ $\la(\om_i)_z,(L_j)_z\ra_z=\delta_{ij}$. Now, let $ \bar L_1,\cdots,\bar L_n,$ $ \bom_1,\cdots,\bom_n$ be the conjugated respectively to $L_1,\cdots, L_n$ and to $\om_1 ,\cdots, \om_n$ on $U$. It turns out that $ \bom_1,\cdots,\bom_n$ is the dual to $\bar L_1,\cdots,\bar L_n$. Note that on $U\cap b\Om$ we have 
$L_j(r)=\bar L_j(r)=\delta_{jn}$; in particular $L_1,\cdots, L_{n-1}$ and $ \bar L_1,\cdots,\bar L_{n-1}$ are local basis for $T^{1,0} b\Om:=\C T^{\C} b\Om\cap T^{1,0}\C^n$ and $T^{0,1}b\Om:=\C T^{\C}b\Om\cap T^{0,1}\C^n$ respectively, where $\C T^{\C} b\Om$ is the complexification of the complex tangent bundle to $b\Om$. We now define the tangential purely imaginary vector field $T$ on $U\cap b\Om$ by $T=L_n-\bar L_n.$  \index{Notation ! $T$} Note that $L_1,\cdots, L_{n-1},\bar L_1,\cdots,\bar L_{n-1}, T$ constitute a local basis of $\C T b\Om$ over $U\cap b\Om$. 
Recall that the Levi form of the function $r$ is the Hermitian form given by $\partial\bar{\partial}r$, and the Levi form of the boundary is $\partial\bar{\partial}r|_{T^\C b\Om}$. The following lemma establishes a useful relation between the components $(r_{ij})$ of the Levi form of $r$ in the adapted frame we have just defined and the commutators between the coordinate vector fields.
 \begin{lemma}\Label{cartan}
\begin{equation} \label{rij}
[L_i,\bar L_j]= r_{ij}T+\sum_{h=1}^{n-1}c^h_{ij}L_h-\sum_{h=1}^{n-1}\bar c^h_{ji}\bar L_h .
\end{equation}
for some complex-valued functions $c^h_{ij}$
\end{lemma}

\begin{proof}
By Cartan's formula, and from the trivial identity $\partial \bar{\partial}=-\bar{\partial}\partial$, we have $\la\om_n,[L_i,\bar L_j]\ra=-\la d\partial r,L_i \wedge L_j\ra + L_i\la\partial r,L_j\ra - \bar L_j \la\partial r,L_i\ra=\la\partial \bar{\partial} r,L_i \wedge \bar{L_j}\ra = r_{ij}$. Analogously $\la\bar{\om_n},[L_i,\bar L_j]\ra=-\la\partial \bar{\partial} r,L_i \wedge \bar{L_j}\ra = -r_{ij}$. Thus we have $[L_i,\bar L_j]=r_{ij}(L_n-\bar L_n)+\cdots$, where the dots denote combinations of $L_h$ and $\bar L_h$ for $1 \le h \le n-1$. We now define $c^h_{ij}=\la\om_h,[L_i,\bar L_j]\ra$, that is $c^h_{ij}$ is the component along $L_h$ of $[L_i,\bar L_j]$, and $d^h_{ij}=\la\bar\om_h,[L_i,\bar L_j]\ra$, that is, $d^h_{ij}$ is the component along $\bar L_h$ of $[L_i,\bar L_j]$. Now observe that $d^h_{ij}=\la\bar\om_h,[L_i, \bar L_j]\ra=\overline{\la\om_h,[\bar L_i, L_j]\ra}=-\overline{\la\om_h,[ L_j, \bar L_j]\ra}=-\bar c^h_{ji}$. The proof is thus complete. \end{proof}
It is useful to observe that the $c^h_{ij}$'s here defined are in fact the components of $\partial\bar\om_h$ in the basis for the $(0,2)$-forms given by $\{\om_i \wedge \bar\om_j\}_{i<j}$.

From now on we denote  by $\Lambda^k$ the space of $(0,k)$-forms in $C^\infty(\bar\Om)$  and  by $C^\infty_c( U\cap\bar\Om)^{k}$
 those which have compact support in $U$. 
  If $u\in C^\infty_c(U\cap \bar\Om)^{k}$, then $u$ can be written as 
\begin{equation}\Label{2u}
u=\sumJ u_J\bom_J,
\end{equation}
where   $\sum'$ denotes summation over strictly increasing indices $J=j_1<...<j_k$ and where $\bar\omega_J$ denotes the wedge product $\bom_J=\bom_{j_1}\we...\we\bom_{j_k}$.  When the multiindices are not ordered, the coefficients are assumed to be alternant. Thus, if $J$ decomposes as $J=jK$, then $u_{jK}=\epsilon^{jK}_{J} u_J$ where $\epsilon^{jK}_{J}$ is the sign of the permutation $jK\simto J$.  \index{Notation ! $\sum'$}

There is a well defined Cauchy-Riemann complex
\begin{equation*}
\Lambda^{k-1}\overset{\bar\partial}\to \Lambda^{k}\overset{\bar\partial}\to \Lambda^{k+1}.
\end{equation*}
where the action of $\bar\partial$ on a $(0,k)$ form is \index{Operator ! Cauchy-Riemann operator, $\dib$}
 \begin{equation}\Label{2dib}
\dib u=\sumJ\sum_{j=1}^n \bar L_j u_J\bom_j\we\bom_J+...
 \end{equation}
whith the dots referring to terms of order zero in $u$. \index{Notation ! $\dib$}
\\

We  extend this complex to $L_2^{0,k}(\Om)$ the  space of $(0,k)$-forms with $L_2$-coefficients, so that the Hilbert space techniques may be applied. For $ u,v \in L_2^{0,k}(\Om)$, we define the inner product and the norm by
$(u,v)=\sum_K^{'} (u_K,v_K)_{L^2}$, $\no u^2=(u,u)$.
For each form of degree $(0,k)$, we define
$$\T{Dom}(\dib)=\{ v\in L_2^{0,k}(\Om) : \dib v \T{ (as distribution)}\in L_2^{0,k+1}(\Om)\}.$$
Then the operator $\dib : \T{Dom}(\dib)\to L_2^{0,k+1}(\Om) $ is well-defined and, by noticing that $\A^{0,k}\subset \T{Dom}(\dib)$, we have $\dib :L_2^{0,k}(\Om)\to L^{0,k+1}_2(\Om) $  as a densely defined operator. Thus, the operator $\dib$ has an $L_2$-adjoint $\dib^*$, defined as follows. If $u\in L_2^{0,k}(\Om), u \in \T{Dom}(\dib^*)$, $\dib^*u \in L_2^{0,k-1}$ is such that
$$(v,\dib^*u)=(\dib v, u)   \T{  for all   }  v\in L_2^{0,k-1}, v \in {Dom}(\dib).$$ 
We have
\begin{eqnarray}
\begin{split}\Label{dom}
(\dib v,u)=&\sumK\sum_{j=1}(\bar L_jv_K,u_{jK})+\dots\\
=&\sumK\sum_{j=1}\Big(-(v_K,L_ju_{jK})+\delta_{jn} \int_{b\Om}v_K\bar u_{jK}dS \Big)+\dots\\
=&\sum_{j=1}\left(v,-\sumK L_ju_{jK}\bom_K)+\delta_{jn}\sumK \int_{b\Om}v_K\bar u_{jK}dS\right)+\dots
\end{split}
\end{eqnarray}
where the second equality follows from integration by parts and the dots denote an error term in which $u$ is not differentiated. If we want $v\longrightarrow (\dib v,u)$ to be a continuous operator on $L_2^{0,k-1}$, then the boundary integral must vanish. This operator is thus represented by an element of $L_2^{0,k-1}$ (Riesz Theorem) that we call $\dib^*u$. We have obtained the proof of the following
\begin{lemma}
\begin{equation}
\Label{2.10*}
u\in \T{Dom}(\dib^*) \T{ if and only if } u_{nK}|_{b\Om}=0 \T{~~for any~~ } K.
\end{equation}\index{Notation ! $\T{Dom}(\dib^*)$}
\end{lemma}
By \eqref{dom} we have 
\begin{eqnarray}\Label{2dib*}
\dib^*u=-{\sumK}\sum_j L_ju_{jK}\bom_K+... 
\end{eqnarray}
\index{Notation ! $\dib^*$}

\section{The basic estimate}
\index{Estimate ! basic }
For a real function $\phi$ in class $C^2$, let the weighted $L^\phi_2$-norm be defined by
$$\no{u}_\phi^2=(u,u)_\phi:=\no{e^{-\frac{\phi}{2}}u}^2=\int_\Om e^{-\phi}\la u,u\ra_zdV$$
where $\la u,u\ra_z=\sum_K u_K(z)\overline{u_K(z)}$

Let $\dib^*_\phi$ be the $L^\phi_2$-adjoint  of $\dib$. It is easy to see that $\T{Dom}(\dib^*)=\T{Dom}(\dib^*_\phi)$ and 
\begin{equation}\Label{2.5}
\begin{split}
\dib_\phi^*u=&- \sumK\sum_{j=1}^{n}\delta_j^\phi u_{jK}\bom_K+\cdots\\
\end{split}
\end{equation}\index{Notation ! $\dib^*_\phi$}
where $\delta^\phi_j u=e^\phi L_j(e^{-\phi}u)= L_j(u)-L_j(\phi)u$ and where dots denote an error term in which $u$ is not differentiated and $\phi$ does not occur.\\
By developing the equalities \eqref{2dib} and \eqref{2.5},  the key technical result is contained in the following 
\begin{theorem}(Morrey-Kohn-H\"ormander)\Label{kmh} Let $z_0\in b\Om$ and $0\le q_o\le n-1$.  Then there exists a neighborhood $U$ of $z_0$ and a suitable constant $C$ such that
\begin{eqnarray} \Label{KMH}
\begin{split}
\no{\bar{\partial} u}^2_{\phi}&+\no{\bar{\partial}^*_\phi u}^2_{\phi}+C\no{u}^2_{\phi}\\
\geq & {\sumK}\sum_{i,j=1}^{n}(\phi_{ij}u_{iK},u_{jK})_\phi-{\sumJ}\sum_{j=1}^{q_o}(\phi_{jj}u_J,u_J)_\phi\\
&+{\sumK}\sum^{n-1}_{i,j=1}\int_{b\Om}e^{-\phi}r_{ij}u_{iK}\bar{u}_{jK}dS-{\sum_{|J|=q}}'\sum^{q_o}_{j=1}\int_{b\Om}e^{-\phi}r_{jj}|u_{J}|^2dS\\
&+\frac{1}{2}\big(\sum^{q_o}_{j=1}\no{\delta_j^{\phi} u}^2_{\phi}+\sum^{n}_{j=q_o+1}\no{\bar{L}_ju}^2_{\phi} \big)
\end{split}
\end{eqnarray}
for any $u\in C^\infty_c(U\cap \bar\Om)^k\cap \T{Dom}(\dib^*)$.\end{theorem}
\begin{proof}
Let $Au$ denote the sum in \eqref{2dib}; we have
\begin{eqnarray}\Label{2.13}
\no{Au}_\phi^2=\sumJ \sum_{j=1}^{n}\no{\bar{L}_ju_J}^2_\phi-\sumK\sum_{ij}(\bar{L}_iu_{jK},\bar{L}_ju_{iK})_\phi.
\end{eqnarray}
 Let $Bu$ denote the sum  in \eqref{2.5}; we have
\begin{eqnarray}\Label{2.14}
\no{Bu}_\phi^2=\sumK\sum_{ij}(\delta^\phi_iu_{iK},\delta^\phi_ju_{jK})_\phi.
\end{eqnarray}
Remember that $Au$ and $Bu$ differ from $\bar\partial u$ and $\bar\partial^*_\phi u$ by terms of order $0$ which do not depend on $\phi$. We then have
\begin{eqnarray}\Label{2.15}
\begin{split}
\no{\bar{\partial} u}^2_{\phi}&+\no{\bar{\partial}^*_\phi u}^2_{\phi}\\
&= \no{Au}_\phi^2+\no{Bu}^2_\phi+R\\ 
&=\sumJ\sum_{j=1}^n\no{\bar{L}_ju_J}^2_\phi+\sumK\sum_{i,j=1}^n(\delta_i^\phi u_{iK},\delta_j^\phi u_{jK})_\phi-(\bar{L}_j u_{iK},\bar{L}_i u_{jK})_\phi +R,
\end{split}
\end{eqnarray}
where $R$ is an error coming from the scalar product of $0$-order terms with terms $\bar L_ju_J$,  $L_ju_{jK}$ or $u$.\\

We  want to apply now integration by parts to the term $(\delta_i^\phi u_{iK},\delta_j^\phi u_{jK})_\phi$ and $(\bar{L}_j u_{iK},\bar{L}_i u_{jK})_\phi$. Notice that for each $u,v \in C_c^1(U\cap \bar{\Omega})$, from integration by parts we have immediately 
$$\begin{cases}
(u,\delta^\phi_j v)_\phi&=-(\bar{L}_ju,v)_\phi+(a_ju,v)_\phi+\delta_{jn}\int_{b\Omega}e^{-\phi}u\bar{v}dS\\
-(u,\bar{L}_iv)_\phi&=(\delta^\phi_iu,v)_\phi-(b_iu,v)_\phi-\delta_{in}\int_{b\Omega}e^{-\phi}u\bar{v}dS
\end{cases}$$
for some $a_j,b_i\in C^1(\bar{\Omega}\cap U)$ independent of $\phi$.

This implies 
\begin{eqnarray}\Label{2.16}
\begin{cases}
(\delta_i^\phi u_{iK},\delta_j^\phi u_{jK})_\phi&=-(\bar{L}_j\delta^\phi_i u_{iK}, u_{jK})_\phi+\delta_{jn}\int_{b\Omega}e^{-\phi}\delta^\phi_i(u_{iK})\bar{u}_{jK}dS+R\\
-(\bar{L}_j u_{iK},\bar{L}_i u_{jK})_\phi&=(\delta^\phi_i\bar{L}_ju_{iK},u_{jK})_\phi-\delta_{in}\int_{b\Omega}e^{-\phi}L_j(u_{iK})\bar{u}_{jK} dS+R.
\end{cases}
\end{eqnarray}
From here on, we denote by $R$ terms involving the product of $u$ by $\delta^\phi_ju$ for $j\le n-1$ or $\bar{L}_ju$ for $j\le n$. \\

Recall that $ u_{nK}|_{b\Omega}\equiv0$ and $L_j(u_{nK})|_{b\Omega}\equiv0$  if $j\le n-1$.  We thus conclude that the boundary integrals vanish in both equalities of \eqref{2.16}. Now, by taking the sum of the two terms in the right side of \eqref{2.16}, after discarding the boundary integrals, we put in evidence the commutator $[\delta^\phi_i, \bar{L}_j]$
\begin{eqnarray}\Label{2.16b}
(\delta_i^\phi u_{iK},\delta_j^\phi u_{jK})_\phi-(\bar{L}_j u_{iK},\bar{L}_i u_{jK})_\phi=([\delta^\phi_i, \bar{L}_j]u_{iK},u_{jK})_\phi+R.
\end{eqnarray}
 \\

Notice that \eqref{2.16} is also true if we replace both $u_{iK}$ and $u_{jK}$ by $u_J$ for indices $i=j\le q_0$. Then we obtain
\begin{eqnarray}\Label{2.17}
\no{\bar{L}_j u_{J}}^2_\phi=\no{\delta^\phi_ju_J}^2_\phi-([\delta^\phi_j, \bar{L}_j]u_J,u_J)_\phi+R.
\end{eqnarray}

Applying \eqref{2.16b} and \eqref{2.17} to the last line in \eqref{2.15}, we have
\begin{equation}\Label{2.19}
\begin{split}
\no{\bar{\partial} u}^2_{\phi}&+\no{\bar{\partial}^*_\phi u}^2_{\phi}\\
&={\sumK}\sum_{i,j=1}^{n}([\delta^\phi_i, \bar{L}_j]u_{iK},u_{jK})_\phi-{\sumJ}\sum_{j=1}^{q_0}([\delta^\phi_j, \bar{L}_j]u_J,u_J)_\phi\\
&+{\sumJ} \Big(\sum_{j=1}^{q_0}\no{\delta^\phi_ju_J}^2_\phi+\sum_{j=q_0+1}^{n}\no{\bar{L}_j u_{J}}^2_\phi\Big)+R.
\end{split}
\end{equation} 
In what follows we use the notation $c^n_{ij}, \bar c^n_{ji}$ to indicate the component of the commutator $[L_i,\bar L_j]$ along $L_n$ and $\bar L_n$ respectively. As we have already seen in \ref{rij} we have $c^n_{ij}=\bar c^n_{ji}=r_{ij}$.
Moreover we denote as $\phi_{ij}$ the coefficients of $\partial\bar\partial\phi$ in the basis $\{\om_i\wedge \bar\om_j\}$, and use the straightforward identity
$\phi_{ij}=L_i\bar L_j(\phi)+\sum_{k=1}^{n}\bar c^k_{ji}\bar L_k(\phi)=\bar L_jL_i(\phi)+ \sum_{k=1}^{n} c^k_{ij} L_k(\phi)$ .\\

Now we calculate the commutator $[\delta^\phi_i, \bar{L}_j]$,
\begin{equation}\Label{2.20}
\begin{split}
[\delta^\phi_i, \bar{L}_j]u &=[L_i-L_i(\phi),\bar L_j]u \\
&=\bar L_j L_i(\phi)+[L_i,\bar L_j]u \\
&=\bar L_j L_i(\phi)+\sum_{k=1}^{n}{c}^k_{ij}L_k(u)- \sum_{k=1}^{n}\bar{c}^k_{ji}\bar L_k(u)\\
&=\bar L_j L_i(\phi)+\sum_{k=1}^{n}{c}^k_{ij}L_k(\phi)u-\sum_{k=1}^{n}{c}^k_{ij}L_k(\phi)u+\sum_{k=1}^{n}{c}^k_{ij}L_k(u)- \sum_{k=1}^{n}\bar{c}^k_{ji}\bar L_k(u)\\
&=\phi_{ij}u+\sum_{k=1}^{n}{c}^k_{ij}L_k(u)-L_k(\phi)u)-\sum_{k=1}^n\bar{c}^k_{ji}\bar{L}_k\\
&=\phi_{ij}u+\sum_{k=1}^{n}{c}^k_{ij}\delta^\phi_k(u)-\sum_{k=1}^n\bar{c}^k_{ji}\bar{L}_k(u)\\
&=\phi_{ij}u+r_{ij}(\delta^\phi_n(u)-\bar L_n(u))+\sum_{k=1}^{n-1}{c}^k_{ij}\delta^\phi_k(u)-\sum_{k=1}^{n-1}\bar{c}^k_{ji}\bar{L}_k(u)
\end{split}
\end{equation} 

Since $\langle L_n,\partial r\rangle=1$, we have
\begin{equation}\Label{2.21}
(r_{ij}\delta^\phi_nu_{iK}, u_{jK})_\phi=\int_{b\Omega}e^{-\phi}r_{ij}u_{iK}u_{jK}dS+\int_{\Omega}e^{-\phi}r_{ij}u_{iK}\bar{L_n}u_{jK}dV+\dots
\end{equation} 
and the integral over $\Omega$ is an error of type $R$.
Substituting \eqref{2.20} in \eqref{2.19} and combining with \eqref{2.21}, we get
\begin{equation}\Label{2.22}
\begin{split}
\no{\bar{\partial} u}^2_{\phi}&+\no{\bar{\partial}^*_\phi u}^2_{\phi}\\
&={\sumK}\sum_{i,j=1}^{n}(\phi_{ij}u_{iK},u_{jK})_\phi-{\sumJ}\sum_{j=1}^{q_0}(\phi_{jj}u_J,u_J)_\phi\\
&+{\sumK}\sum_{i,j=1}^{n-1}\int_{b\Omega}r_{ij}u_{iK}\bar{u}_{jK}e^{-\phi}dS-{\sumJ}\sum_{j=1}^{q_0}\int_{b\Omega}r_{jj}|u_J|e^{-\phi}dS\\
&+{\sumJ} \Big(\sum_{j=1}^{q_0}\no{\delta^\phi_ju_J}^2_\phi+\sum_{j=q_0+1}^{n}\no{\bar{L}_j u_{J}}^2_\phi\Big)+R.
\end{split}
\end{equation} 

We denote by $S$ the sum in the last line in \eqref{2.22}. To conclude our proof, we only need to prove that 
for a suitable $C$ independent of $\phi$ we have 
\begin{eqnarray}\Label{2.23}
R\le \frac{1}{2}{\sumJ} \Big(\sum_{j=1}^{q_0}\no{\delta^\phi_ju_J}^2_\phi+\sum_{j=q_0+1}^{n}\no{\bar{L}_j u_{J}}^2_\phi\Big)+C\no{u}_\phi^2.
\end{eqnarray}
In fact, if we point our attention to those terms which involve $\delta^\phi_ju$ for $j\le q_o$ or $\bar{L}_ju$ for $q_o+1\le j\le n$, then \eqref{2.23} is clear simply by using Cauchy-Schwartz inequality, since $S$ carries the corresponding squares $\no{\delta^\phi_ju}^2_\phi$ and $\no{\bar{L}_ju}^2_\phi$. Otherwise, we note that for $j\le n-1$ we may interchange $\bar{L}_j$ and $\delta^\phi_j$ by means of integration by parts: boundary integrals do not occur because $L_j(r)=0$ on $b\Omega$ for $j\le n-1$. As for $\delta^\phi_n$, notice that it only hits coefficients whose index contains $n$ and hence $u_{nK}=0$ on $b\Omega$. So $\delta^\phi_n(u_{nK})\bar{u}$ is also interchangeable with $u_{nK}\bar{L}_n\bar{u}$ by integration by parts. This concludes the proof of Proposition \ref{kmh}.
\end{proof}
For the choice $\phi=0$, we can rewrite the estimate \eqref{KMH} as 
\begin{equation}\Label{aaa} 
\no{\bar{\partial} u}^2+  
\no{\bar{\partial^*} u}^2+  
\no{u}^2  \simge  {\sumK}\sum^{n-1}_{i,j=1}\int_{b\Om}r_{ij}u_{iK}\bar{u}_{jK}dS-{\sum_{|J|=q}}'\sum^{q_o}_{j=1}\int_{b\Om}r_{jj}|u_{J}|^2dS
+\sum^{q_o}_{j=1}\no{L_j  u}^2+\sum^{n}_{j=q_o+1}\no{\bar{L}_ju}^2
\end{equation}
for any $u\in C^\infty_c(U\cap \bar\Om)^k\cap\T{Dom}(\dib^*)$.\\

Observe that for $u \in C^\infty_c(U\cap \bar\Om)$, $j\le n-1$, each $\no{L_j u}^2$ can be interchanged with $\no{\bar L_j u}^2+R$. If $u|_{b\Om}\equiv 0$, then this is true even for $j=n$, due to the vanishing of the boundary integral.
Thus, for $u \in C^\infty_c(U\cap \bar\Om)$, $u|_{b\Om}\equiv 0$,
\begin{equation}\Label{index1}
\begin{split}
\sum^{q_o}_{j=1}\no{L_j u}^2+&\sum^{n}_{j=q_o+1}\no{\bar{L}_j u}^2+C\no{u}^2\\
\ge & \frac{1}{2}\sum^{n}_{j=1}\big(\no{L_ju}^2+\no{\bar{L}_ju}^2\big) +\no{u}^2\\
\simge &\no{u}_1^2
\end{split}
\end{equation}
where $\no{.}_1$ is the Sobolev norm of index 1.\\

In conclusion, combining \eqref{aaa} and \eqref{index1}, and observing that the boundary integrals are zero, we get an estimate which fully expresses  the interior elliptic regularity of the system $(\dib, \dib^*)$, and is known as Garding inequality:
\begin{equation}\Label{Garding}
\no{u}_1^2\simleq   
\no{\bar{\partial} u}^2+  
\no{\bar{\partial^*} u}^2+  
\no{u}^2,  \ \   u \in C^\infty_c(U\cap \bar\Om), \ \ u|_{b\Om}\equiv 0 
\end{equation}

In general, if $u|_{b\Om}\ne0$, to take full advantage of \eqref{aaa} we need to be able to control the integrals at the boundary. In order to achieve this we introduce the geometrical notion of $q$-pseudoconvexity. \\

Let $\lambda_1(z)\le ...\le \lambda_{n-1}(z)$ be the eigenvalues of the Levi form of the boundary $(r_{jk}(z))^{n-1}_{j,k=1}$ and denote by $s^+_{b\Om}(z)$, $s^-_{b\Om}(z)$, $s^0_{b\Om}(z) $ their number according to the different sign. \\

 We take a pair of indices $1\le q\le n-1$ and $0\le q_o\le n-1$ such that $q\not=q_o$. We assume that there is a bundle $\mathcal V^{q_o}\in T^{1,0}b\Om$ of rank $q_o$ with smooth coefficients that, by reordering, we may suppose to be the bundle $\mathcal V^{q_o}=\T{ span }\{ L_1,...,L_{q_o}\}$, such that 
\begin{equation}\Label{2qpseu}
\sum_{j=1}^{q}\lambda_j(z)-\sum_{j=1}^{q_o}r_{jj}(z)\ge 0 \qquad z \in  U\cap b\Om.
\end{equation} 
We  say $\Om$ is $q$-pseudoconvex or $q$-pseudoconcave according to $q>q_o$ or $q<q_o$. 
\begin{lemma}
Condition \eqref{2qpseu} is equivalent to
\begin{equation}
\sumK\sum_{ij=1}^{n-1}r_{ij}u_{iK}\bar u_{jK}-\sum_{j=1}^{q_o}r_{jj}|u|^2\ge 0 \T{ on } U\cap b\Om
\end{equation}
for any $u\in C_c^\infty(U\cap\bar\Om)^k\cap \T{Dom}(\dib^*)$.
\end{lemma}
\begin{proof} The proof of the Lemma immediately follows from the estimate
$$
\sumK\sum_{ij=1}^{n-1}r_{ij}u_{iK}\bar u_{jK}\geq\underset{j=1}{\overset q\sum}\lambda_j|u|^2,
$$
for any $u\in C_c^\infty(U\cap\bar\Om)^k\cap \T{Dom}(\dib^*)$ with equality for a suitable $u$. In turn, the proof of this estimate is obtained by diagonalizing the matrix $(r_{ij})$ (cf. \cite{H65} and \cite{C87}).
\end{proof}

Note that \eqref{2qpseu} for $q>q_o$  implies $\lambda_{q}\ge 0$; hence \eqref{2qpseu} is still true if we replace the first sum $\sum^{q}_{j=1}\cdot$ by $\sum^{k}_{j=1}\cdot$ for any $k$ such that $q\le k\le n-1$. Similarly, if it holds for $q<q_o$, then $\lambda_{q+1}\le0$ and hence it also holds with $q$ replaced by $k\le q$ in the first sum.\\

We notice  that $q$-pseudoconvexity/concavity is invariant under a change of an orthonormal basis but not of an adapted frame. In fact, not only the number, but also the size of the eigenvalues comes into play. Thus, when we say that $b\Om$ is $q$-pseudoconvex/concave, we mean that there is an adapted frame in which \eqref{2qpseu} is fulfilled. Sometimes, it is more convenient to put our calculations in an orthonormal frame. In this case, it is meant that the metric has been changed so that the adapted frame has become orthonormal.\\
   
The following theorem is a straightforward application of the Morrey-Kohn-H\"ormander formula to the case of a $q$-pseudoconvex domain.    
\bt
\Label{t1.9.1}
Let $\Omega$ be $q$-pseudoconvex; then, for $\phi_t:=(t+C)|z|^2$ and for any $u\in C^\infty(\bar\Omega)^k\cap D_{\bar\partial^*}$, we have
\begin{equation}
\Label{1.9.20}
t||u||^2_{\phi}\leq ||\bar\partial u||^2_{\phi}+||\bar\partial_{\phi}^*u||^2_{\phi}\quad\T{ if } k\geq q+1.
\end{equation}
\et
\begin{proof} 
We choose $q_o$ as in \eqref{2qpseu}; we have
\begin{equation}
\Label{1.9.21}
\begin{cases}
\sumK\sumij r_{ij}u_{iK}\bar u_{jK}-\sumJ\,\sumjq r_{jj}|u_J|^2\geq0,
\\
\sumK\sumij\phi_{ij}u_{iK}\bar u_{jK}-\sumJ\,\sumjq\phi_{jj}|u_J|^2\geq (k-q_o)(t+C)|u|^2.
\end{cases}
\end{equation}
Using \eqref{KMH} and noticing that $k-q_o\geq 1$, we get \eqref{1.9.20}. 
\end{proof}
Then, if we define the weighted energy form $Q^{\phi}(u,v):=(\bar\partial u,\bar\partial v )_{\phi}+(\bar\partial^*_{\phi} u,\bar\partial^*_{\phi} v )_{\phi}$ for any $u,v \in C^\infty(\bar\Omega)^k\cap D_{\bar\partial^*}$, what we have in fact proved above is that for any  $u\in C^\infty(\bar\Omega)^k\cap D_{\bar\partial^*_{\phi}}$ we have 
\begin{equation}\Label{989}
t \no{u}_{\phi}^2 \le Q^{\phi}(u,u). 
\end{equation}
By Riesz theory (cf. \cite{CS01}) it is possible to show that \eqref{989} implies that $\bar\partial$ and $\bar\partial^*$ have closed range, that is  
\begin{equation}\Label{basicunweight}
\no{u}^2 \lesssim Q(u,u) 
\end{equation}
or, equally, 
\begin{equation}\Label{boxbasic}
\no{u}^2 \lesssim \no{\Box u}^2.
\end{equation}
In particular, if $\Om$ is $q$-pseudoconvex and  $u \in C^\infty_c(U\cap \bar\Om), u|_{b\Om}\equiv 0$, 
\eqref{boxbasic} can be improved to
 \begin{equation}\Label{Gardingconvex}
\no{u}_2^2 \lesssim ||\Box u||^2
\end{equation}

\chapter{Subelliptic and superlogarithmic estimates}
In this chapter, we introduce subelliptic estimates and show how they can be used to establish hypoellipticity at the boundary for the $\bar\partial$-Neumann problem. Then we see under which conditions on the geometry of the domain such estimates hold.

\section{Subellipticity and Hypoellipticity}
We introduce subelliptic estimates for a general system of smooth real vector fields $L_1,...,L_n$.
 Let $\F$ be the Fourier transform defined, over functions $u$ in the
 Schwartz space $\mathcal S$ of rapidly decreasing $C^\infty$
 functions, by $\F u(\xi):=\int_{\R^n}e^{-i\langle
   x,\xi\rangle}u(x)dx$. (Sometimes, we write $\hat u(\xi)$ instead of $\F u(\xi)$). Let $\Lambda^s$ be the standard elliptic pseudodifferential operator of (non necessarily integer) order $s$ defined by $\Lambda^su=\F^{-1}\left((1+|\xi|^2)^{\frac s2}\hat u(\xi)\right)$. We introduce a scalar product by
\begin{equation*}
\begin{split}
\langle u,v\rangle_{H^s}&=
 \langle \Lambda^s u,\Lambda^sv\rangle_{H^0}
 \\
 &=\langle (1+|\xi|^2)^s\hat u,\hat v\rangle_{H^0},
 \end{split}
 \end{equation*}
  and define $H^s$ to be the completion of $\mathcal S$ under the
  associated norm. We denote by $\langle\cdot,\cdot\rangle_{H^s}$ and
  $||\cdot||_{H^s}$ the scalar product and the norm in $H^s$,
  respectively. By the Plancherel theorem we have, for integer $s$, $H^s=\{u:\,\,D^\alpha u\in H^0,\,\,|\alpha|\leq s\}$ with $||u||^2_{H^s}\simeq\underset{|\alpha|\leq s}\sum||D^\alpha u||^2_{H^0}$. \\
  We are now ready to define subellipticity for the system $L_1,...,L_n$. Let $\epsilon>0$.
  
  \bd
  The system $L_1,...,L_k$ is said to be ``$\epsilon$-subelliptic'' at $x_o$ when, for a neighborhood $B$ of $x_o$,
  \begin{equation}
  \Label{2.5.3}
  ||u||^2_{H^\epsilon}\simleq \underset{j=1,...,k}\sum||L_ju||^2_{H^0}+||u||^2_{H^0}\quad \T{ for any } u\in C^\infty_c(B).
  \end{equation}
  \ed
  \bd
  The system $L_1,...,L_k$ is said to be ``hypoelliptic'' at $x_o$ 
  when for a neighborhood $B$ of $x_o$ and for any $x\in B$,
  \begin{equation}
  \Label{2.5.4}
  u\in H^0,\,\,L_ju\in C^\infty_{x},\,\, j=1,...,k\quad\Rightarrow
  \quad u\in C^\infty_{x}, 
  \end{equation}
where $C^\infty_x$ denotes the space of germs of $C^\infty$-functions
at $x$.
  \ed
  To prove that subellipticity implies hypoellipticity we need a generalisation of the famous Friedrichs' Lemma to the case of a general Sobolev space $H^s$. We state this result below, and the proof is to be found in \cite{Z08}(revised).  
  
  \bl (Friedrichs')
  \Label{Friedrichs}
  Let $\chi_\nu$'s be a sequence of smooth functions which approximate the Dirac measure. Let $g\in H^s$ have compact support and suppose $Lg\in H^s$.  Then,
  \begin{equation}
  \Label{pinton1}
  L(g*\chi_\nu)\to Lg\quad\T{ in $H^s$}.
  \end{equation}
  \el
With this lemma on hand, we can now easily prove the following important result, which justify our interest in subellipticity. 
\bt
  \Label{t2.5.1}
 If $L_1,...,L_k$ is subelliptic, then it is also hypoelliptic.
 \et
 \begin{proof} 
 (i): Let $\psi$ be a smooth cut-off function such that $\psi\equiv1$ at $x$. We show that $\psi u\in H^0$ and $L_j(\psi u)\in H^0\,\,\T{ for any } j$ implies $\psi u\in H^\epsilon$. In fact, let us  approximate $\psi u$ by $u_\nu:=\chi_\nu*(\psi u)\in C^\infty_c$; by Lemma~\ref{Friedrichs} applied with $s=0$ and $g=\psi u$, we have that  $u_\nu\to \psi u$ and $L_ju_\nu\to L_j(\psi u)$ in $H^0$-norm. From the subellipticity of $L_1,...,L_k$ we have
 \begin{equation*}
 ||u_\nu-u_\mu||^2_{H^\epsilon}\simleq\sum_j||L_ju_\nu-L_ju_\mu||^2_{H^0}+||u_\nu-u_\mu||^2_{H^0},
\end{equation*}
from which we deduce that $\{u_\nu\}_\nu$ is Cauchy in $H^\epsilon$ and therefore its limit $\psi u$ is also in $H^\epsilon$.

(ii): Let $L_j( u)\in H^s\,\,\T{ for any } j$ over $\T{supp}\,\psi$ and suppose that we already know that $ u$ belongs to $ H^\sigma$ over $\T{supp}\,\psi$ for $\sigma\leq s$: we wish to prove that, in fact, $\psi u\in H^{\sigma+\epsilon}$. Let $\psi^1\in C^\infty_c$ with $\psi^1\equiv1$ over $\T{supp}(\psi)$; we begin by proving that $L\psi^1\Lambda^\sigma\psi u\in H^0$. In fact, we have the elementary equality
$$
L\psi^1\Lambda^\sigma u_\nu=\Lambda^\sigma L\psi^1u_\nu+[L,\psi^1\Lambda^\sigma]u_\nu.
$$
Now, the first term in the right is $H^0$ convergent, for $\nu\to\infty$, to $\Lambda^\sigma L\psi u$; this follows from Lemma~\ref{Friedrichs}. The second is trivially $H^0$ convergent to $[L,\psi^1\Lambda^\sigma]\psi u$. Thus, the first in the left is also $H^0$ convergent and, by $H^{-1}$ uniqueness, it must converge to $L\psi^1\Lambda^\sigma\psi u$; so this belongs to $H^0$.
 Hence, we apply (i) to $\psi^1\Lambda^\sigma(\psi u)$
 and conclude that $\psi u\in H^{\sigma+\epsilon}$. Iterated use of this argument, with a gain of Sobolev index $\epsilon$ and a shrinkig of $\T{supp}(\psi)$ at each step, makes it possible to conclude that $\psi u\in H^s$.
 (Note here that the cut off in next step must have support where the former is $\equiv1$.) 

\end{proof}  
  
\section{Conditions for Subellipticity of the \texorpdfstring{$\bar\partial$}{dbar}-Neumann problem}
 We prove $\frac{1}{2}$-subelliptic estimates in the case of a strictly $q$-pseudoconvex domain, that is, when \eqref{2qpseu} holds with the inequality ``$\geq0$'' replaced by ``$\geq c$'' for $c>0$. This proof can be obtained in two different ways. One relies on the fact that strictly $q$-pseudoconvex domains are of ``commutator type'' two, while the other is done using weights. The two proofs reflect two different approaches to the problem, which give rise to more general theories. In the first case it can be proved that if for any $L\in T^{1,0}b\Om$ the iterated commutators of $L$ and $\bar L$ produce the tangential purely imaginary vector field $T$, then we have subellipticity. In the second case, the method of weights can be used to obtain subelliptic estimates in a finite D'Angelo-type domain. Morover, if the boundary is real analytic, we can also use an algorithm introduced by Kohn in \cite{K79}.

We start from the first of the two ways we have just described by recalling the notion of ``commutator type'' for a system of real vector fields. 
\bd
Consider a system of smooth real vector fields $\mathfrak X=\{X_1,...,X_n\}$ of $T \R^N$. 
Define, for every $m\in \mathbb{N}$,
\begin{equation}
\Label{2.5.1} 
\mathcal X^m=\T{Span}\left\{
X_i,\,[X_{i_1},X_{i_2}],...,\underset{j}{\underbrace{[X_{i_1},[X_{i_2},...
[X_{i_j-1},X_{i_j}]...]]}}
\vphantom {X_i,\,[X_{i_1},X_{i_2}],...,\underset{j}{\underbrace{[X_{i_1},[X_{i_2},...
[X_{i_j-1},X_{i_j}]...]]}}}
\T{ for any } i,\,i_1,...,i_j\in\{1,...,n\}\right\}.
\end{equation}
We say the system $\mathfrak X$ is of type $m$ if there exists $m\in \mathbb{N}$ such that $\mathcal X^m=  T\R^N$ and $\mathcal X^{m-1}\ne  T \R^N$.
We say $\mathfrak X$ is of infinite type if there is no such $m$.
\ed
A theorem by Rotschild and Stein (cf. \cite{RS76}) shows that for a system of real vector fields of finite type $m$ we have a subelliptic estimate with a ``gain of derivative'' of $\frac{1}{m}$, and this result is sharp. Therefore, for a system of real vector fields of finite type $m$, subellipticity is expressed by the H\"ormander formula
 \begin{equation}
 ||u||^2_{H^{\frac{1}{m}}}\simleq \underset{j=1,...,n}\sum||X_ju||^2_{H^0}+||u||^2_{H^0}\quad \T{ for any } u\in C^\infty_c.
  \end{equation}
Note that if the system $X_1,...,X_n$ is not of finite type but the Lie span $\mathcal X$ has locally 
constant rank $a<N$, then, in some new variables $\{x'_j\}_{1\le j \le N}$, we have $\mathcal
X=\T{Span}\{\partial_{x'_j}\}_{1\leq j\leq a}$; this follows from the 
Frobenius theorem. Any function $u$ of $x''_j,\,\,j\geq a+1$ only, is then a solution of the system $X_ju=0$, but it is not 
necessarily regular. Thus the sufficient condition stated above is also necessary. We can conclude that, for a system of real vector fields whose Lie span has locally constant rank, being of finite commutator type is equivalent to having a subelliptic estimate.\\

Consider now a system of vector fields $L:= \{ L_j\} \subset T^{1,0}\C^n$. We can artificially achieve stability under conjugation by adding $\epsilon\bar L_j$'s. Suppose now that the stabilized system is of finite type $m$. We can then apply H\"ormander formula and, giving small and large constant, we obtain
 \begin{equation}\Label{to}
 ||u||^2_{H^{\frac{1}{m}}}\le \underset{j=1,...,k}\sum (c_{\epsilon}||L_ju||^2_{H^0}+\epsilon||\bar L_ju||^2_{H^0})+c_{\epsilon}||u||^2_{H^0} \quad u\in C^{\infty}_c.
  \end{equation}
On the other hand, through integration by parts, we have 
 \begin{equation}\Label{in}
 ||\bar L_j u||^2\simleq ||L_j u||^2+|([L_j,\bar L_j]u,u)|+||u||^2 \simleq ||L_j u||^2+||u||^2_{\frac{1}{2}}+||u||^2
  \end{equation}
 Therefore, if the type is $m=2$, inserting \eqref{in} into \eqref{to}, the $\frac{1}{2}$-norm is absorbed
in the left, the ${\epsilon \bar L_j}$'s can be taken back and we obtain a $\frac{1}{2}$-subelliptic estimate for the system $\{L_j\}$. The restraint $m=2$ is substantial: in fact Kohn was able to produce in \cite{K06} a pair of vector fields
$\{L_1, L_2\}$ in $\C \times \R$ of finite type $m$ (for any $m\ge 3$) which are not subelliptic. (Nonetheless,
they are are hypoelliptic). 
We can then conclude that, in the case of complex vector fields, having finite commutator type $m$ is not a sufficient condition for subellipticity, except from the case $m=2$.

We will try to compensate this asymmetry between the real and the complex case by introducing a little adjustment to the notion of ``type'' that will allow us to give a sufficient condition for subellipticity in the case of a system of complex vector fields. In particular we will define the notion of type for a single vector field $L$, and we will show that we have subellipticity when all the vector fields in the system have finite type.

Coming back to the setting of the $\bar\partial$-Neumann problem, recall that we work with an adapted frame of vector fields $L_1,...,L_n,\bar L_1,...,\bar L_n$ defined as in Section 1. In particular, if the boundary of our domain $\Om$ is defined by $r=0$, we have that $L_j(r)=\bar L_j(r)=\delta_{jn}$. We have already observed that $L_1,\cdots, L_{n-1},\bar L_1,\cdots,\bar L_{n-1}, T$ form a local basis of $\C T b\Om$, where $T=L_n-\bar L_n$. Therefore, if  
$L=\{L_1,\cdots, L_{n-1},\bar L_1,\cdots,\bar L_{n-1}\}$ is our system of vector fields, $T$ is the only ``missing direction'' (usually referred to in the literature as the ``bad diretion''). Hence $L$ is of finite type in $\C T b\Om$ if and only if $T$ can be produced through iterated commutators. We can actually say more.

\begin{lemma}\Label{Tellsus}
Let $L$ be of type $m\in \mathbb N$, that is, $m$ is the smallest integer such that $T$ is generated by commutators of $m$ vector fields among $\{L_j, \bar L_j\}$. Then there exists $L\in Span\{L_1,\cdots, L_{n-1},\bar L_1,\cdots,\bar L_{n-1}\}$ such that 
\begin{equation}
\Label{comm} 
T= \underset{m}{\underbrace{[\overset{(-)}L,[\overset{(-)}L,...
[\overset{(-)}L,\overset{(-)}L]...]]}}
\end{equation}
where $\overset{(-)}L$ denotes either occurrence of $L$ or $\bar L$. 
\end{lemma}
This lemma, whose proof relies entirely on linear algebra, serves as a motivation for the definition below.
\bd
We say that a vector field $L$ is of type $m \in \mathbb N$ if $m$ is the smallest integer such that the purely imaginary tangential direction $T$ is generated by a commutator of $\overset{(-)}L$'s of order $m$. 
\ed

What Lemma~\ref{Tellsus} tells us is that the complex system $\{L_j\}_{1\le j \le k}$ is of finite type $m$ if and only if there exists $L\in$ Span$\{L_1,\cdots, L_k\}$ of lowest type $m$. We have already pointed out that this is not a sufficient condition for subellipticity. As we show in the following theorem, in order to ensure subellipticipity we assume that $L_j$ is of finite type for all $j$.
Before stating the Theorem we recall the tangential Sobolev norm, defined as $|||u|||_s:=||\Lambda^s u||$, and the relation $||u||_s=\sum_{j=1}^{[s]+1}|||\partial_r^ju|||_{s-j}$. We have the two following important fact:
\begin{equation}
|||u|||_\epsilon=\sum_{j=1}^{n}|||\bar L_ju|||_{-1+\epsilon}+\sum_{j=1}^{n-1}|||L_ju|||_{-1+\epsilon}+|||Tu|||_{-1+\epsilon}
\end{equation}
\begin{equation}
|||\partial_ru|||_{-1+\epsilon}\le |||\bar L_nu|||_{-1+\epsilon}+|||u|||_\epsilon.
\end{equation}
Moreover, recall that $\sum_{j=1}^{n}||\bar L_ju||^2+\sum_{j=1}^{n-1}||L_ju||^2\le Q(u,u)$.
Combining these relations we can understand why, in the $\bar\partial$-Neumann problem, subelliptic estimates usually appear as follows:
\begin{equation}
\Label{1.10.3,1}
|||T u|||^2_{-1+\epsilon}\simle 
Q(u,u)+||u||^2_{0}.
\end{equation}
In fact the main point is to get control of the $\epsilon$-norm in the $T$-direction, since the other ones are already controlled by the energy. 
\bt\Label{general}
Let $\{L_j\}_{1\le j \le n-1}$ be a system of complex vector fields such that each $L_j$ is of finite type $m_j$. Then the system enjoys a $\frac{1}{m}$-subelliptic estimate, where $m=\underset{j}max\{ m_j \}$.
\et
\begin{proof}
Let $\epsilon=\frac{1}{m}$. By H\"ormander we have, for any $1\le j \le n-1$, the estimate 
\begin{equation}\Label{hor}
|||Tu|||_ {-1+\epsilon}^2 \lesssim||L_ju||^2_{H^0}+ ||\bar L_ju||^2_{H^0}+||u||^2_0.
\end{equation}  
Exploiting the usual relation
\begin{equation}\Label{dis}
|||L_ju|||_ {-1+\epsilon}^2 \le |||\bar L_ju|||_ {-1+\epsilon}^2 + |\int{\Lambda^{-2+2\epsilon}[L_j,\bar L_j]u\bar u}dV | + ||u||_0^2,
\end{equation}
we first sum on all the $j$'s. We then add $Q(u,u)$ on the right and $|||Tu|||_ {-1+\epsilon}^2 + \sum_j|||\bar L_j u|||_ {-1+\epsilon}^2$ on the left. From \eqref{hor} and the fact that $\sum_j|||\bar L_j u|||_ {-1+\epsilon}^2$ is contained in the energy the inequality is preserved. Note that under the integral in \eqref{dis} we have a tangential operator of order $-1+2\epsilon$. Then, for small $\delta>0$, since $-1+\epsilon \le 0$, we get
\begin{equation}\Label{dis2}
|||Tu|||_ {-1+\epsilon}^2 + \sum_j|||\bar L_j u|||_{-1+\epsilon}^2+\sum_j|||L_ju|||_ {-1+\epsilon}^2 \le  \delta||\Lambda^{\epsilon}u||_0^2 + \delta^{-1}||u||_0^2 + Q(u,u) + ||u||_0^2.
\end{equation}
We conclude that
\begin{equation}
||u||_\epsilon \lesssim Q(u,u) + ||u||_0^2
\end{equation}
\end{proof}

From this general theorem we get as an immediate corollary the result we wanted to prove at the beginning of the section.
\begin{corollary}
Let $\Om$ be a strongly $q$-pseudoconvex domain. Then we have $\frac{1}{2}$-subelliptic estimates for the $\bar\partial$-Neumann problem.
\end{corollary}
\begin{proof}
Since every eigenvalue of the Levi form is positive, by Lemma ~\ref{cartan}, we have that every commutator of the kind $[L_j,\bar L_j]$ generates $T$. Then the $L_j$'s have all type 2. By Theorem ~\ref{general} we then have $\frac{1}{2}$-subelliptic estimates. 
\end{proof}

The second way by which we can obtain this result was originally pursued by H\"ormander in \cite{H65} and exploits a clever choice of weights in the Morrey-Kohn-H\"ormander formula.
In the case of a $q$-strictly pseudoconvex domain we can see that, up to a change of coordinates, the boundary has equation $x_n=|z'|^2+...$, and the domain is $x_n<|z'|^2+...$, where $z=(z',z_n)$ and the dots stand for terms of order 3 or higher. We then consider a family of weights $\{\phi_\delta\}_\delta$ with $\phi_\delta=-\log(\frac{|z'|^2-x_n}{\delta} +1)$. Observing that we have $\partial\bar\partial\phi_\delta>\delta^{-1}$, by plugging the weights into \eqref{kmh} we have that the integral $\int_\Om \delta^{-1}|u|^2dV$ is controlled by the energy. Since $||\delta^{-1/2}u||^2\sim ||u||^2_{H^{\frac{1}{2}}}$ (cf. \cite{CS01}), we obtain the $\frac{1}{2}$-subelliptic estimate for strictly $q$-pseudoconvex domains. We can push this technique further: for a domain with boundary expressed by $x_n=|z'|^{2m}$, using the weights $\{\phi_\delta\}_\delta=\{-\log(\frac{|z'|^{2m}-x_n}{\delta} +1)+\log(\frac{|z|^2} {|\delta|^{\frac{1}{2m}}}+1)\}$, observing that $\partial\bar\partial\phi_\delta>\delta^{-1/m}$ and that $||\delta^{-1/{2m}}u||^2\sim ||u||_{H^{\frac{1}{2m}}}$, we can establish $\frac{1}{2m}$-subelliptic estimates. Exploiting the full power of this method, Catlin proved an ultimate criterion for subellipticity, stating in \cite{C87} that subelliptic estimates hold for $k$-forms at $z_0$ if and only if the D'Angelo type $D_k{z_0}$ is finite.
\\

We point out that subelliptic estimates are included in a more general theory of estimates for the $\bar\partial$-Neumann problem, namely that of the ``$f$-estimates'' (cf. \cite{Kh09}). 
Recall that $\Lambda$ is the pseudodifferential operator whose symbol is $\Lambda_\xi=(1+|\xi|^2)^{1/2}$. With the notation that $f(\Lambda)$ stands for the pseudodifferential operator whose symbol is $f(\Lambda_\xi)$, we say that we have a $f$-estimate if 
 $||f(\Lambda)u|| \lesssim Q(u,u)+||u||^2$ for any $u\in C^\infty_c(U \cap \Om)^k$. Note that subelliptic estimates correspond to a choice of $f(\Lambda_\xi)=|\Lambda_\xi|^{\epsilon}$ for some $o<\epsilon<1$. For a choice of $f(\Lambda_\xi)=\log|\Lambda_\xi|$ we obtain instead superlogarithmic estimates, which we will deal with in the following section. We recall a general result that is related to these estimates, for whose proof see \cite{Kh09} or \cite{KZ12}.
\bt\Label{fest}
Let $\Om$ be a pseudoconvex domain and $r=0$ the local equation of $b\Om$. If $|r|<F(|z'|)$ and we have a $f$-estimate, then, for small $\delta$, $$\frac{f(\delta^{-1})}{\log(\delta^{-1})} \lesssim F^*(\delta )^{-1},$$ where $F^*$ is the inverse function to $F$. 
\et
With this on hand we can give a nice proof of a well known result.
\bt
No subelliptic estimates exist for $\epsilon > 1/2$.
\et

\begin{proof}
From class $C^2$ of the domain we have that $|r|<|z'|^2$. Suppose we have a $\epsilon$-subelliptic estimate. Hence if we invoke Theorem ~\ref{fest} with $F(|z'|)=|z'|^2$ and $f(t)=|t|^{\epsilon}$, for $t=\delta^{-1}$, we get that $\delta^{1/2}\lesssim \delta^\epsilon \log(\delta^{- 1})$ for small $\delta>0$, from which we have $\epsilon \le 1/2$. 
\end{proof}

\section{Superlogarithmic estimates}
We have seen in the previous section how subellipticity implies local hypoellipticity and we have proved (through the notion of type or the use of weights) some conditions for subellipticity of the $\bar\partial$-Neumann problem. There are cases in which subelliptic estimates do not hold but the $\bar\partial$-Neumann problem is nevertheless hypoelliptic. Works by Christ \cite{Ch00}\cite{Ch01}\cite{Ch02}, Kusuoka and Stroock \cite{KS}
provided examples in which local regularity is governed by a weaker kind of estimates, called superlogarithmic. 
In this section we give the definition of superlogarithmic estimates and then we prove how they imply hypoellipticity for the $\bar\partial$-Neumann problem. For the proof we follow the outline of \cite{K02}, where Kohn developed the theory of superlogarithmic estimates and proved this result for the tangential problem.

We should point out that in Chapter 8 Kohn switches his attention to the $\bar\partial$-Neumann and proves that superlogarithmic estimates for the tangential problem imply hypoellipticity for the $\bar\partial$-Neumann problem. However, our approach is a bit different, since we start and end up with the $\bar\partial$-Neumann. Moreover, our result is not a direct consequence of Kohn's theorem, since it is not possible to transfer a superlogarithmic estimate from the tangential system to the $\bar\partial$-Neumann problem. In fact in general (see Khanh) for an $f$-estimate we have a logarithmic loss in passing from the tangential setting to the $\bar\partial$-Neumann, so that subelliptic estimates can be transferred with an arbitrarily small loss, but superlogarithmic ones cannot. 

\bd
We say that a superlogarithmic estimate holds in a neighborhood $U$ of a boundary point if for every $\delta>0$ there exists $C_\delta$ such that 
\begin{equation}\Label{superlog}
||(\log\Lambda)u||^2 \le \delta Q(u,u)+C_\delta||u||^2
\end{equation}
for any $u\in C^\infty_c(U \cap \Om)^k$
\ed  

\bt
Suppose \eqref{superlog} holds in a neighborhood $U$ of a boundary point. Then we have local boundary regularity for the $\bar\partial$-Neumann problem. 
\et

\begin{proof}
We have to prove that if $u$ is a square integrable $(p,q)$-form such that $\Box u=\alpha$, with $\alpha$ square integrable whose restriction to $U$ is $C^\infty(U)$, the restriction of $u$ to $U$ is also in $C^{\infty}(U)$. More precisely, we will prove that, for any $\zeta_0, \zeta_1\in C_c^{\infty}(U)$ with $\zeta_0\prec\zeta_1$, we have that, for any $s>0$, if $\zeta_1\alpha\in H^s$ then $\zeta_0 u\in H^s$. To do this we prove the following estimate: given $s>0$, there exists $C_s>0$ such that 
\begin{equation}\Label{resultfake}
||\zeta_0 u||_s \le C_s(||\zeta_1\alpha||_s+||u||)
\end{equation}
for all $(p,q)$-forms $u$ with smooth coefficients in $U$. 
We will in fact prove that
\begin{equation}\Label{result}
|||\zeta_0 u|||_s \le C_s(||\zeta_1\alpha||_s+||u||).
\end{equation}
How to pass from \eqref{result} to \eqref{resultfake} is standard literature (see for example \cite{CS01}).\\
Let $\sigma\in C_c^{\infty}(U)$ such that $\zeta_0\prec\sigma\prec\zeta_1$. We define the pseudodifferential operator $R^s$ by
\begin{equation}
R^s u(x,r)=\int e^{ix\cdotp \xi}(1+|\xi|^2)^{\frac{s\sigma(x)}{2}}\hat u(\xi,r)d\xi,
\end{equation}
for $u\in C_c^{\infty}(U)$.
Since $\sigma=1$ in the support of $\zeta_0$, the symbol of $(\Lambda^s-R^s)\zeta_0$ is zero, and we have
\begin{equation}\Label{com0}
\begin{split}
|||\zeta_0u|||_s\le ||R^s(\zeta_0u)||+C||u||&=||R^s(\zeta_0\zeta_1u)||+C||u||\\
&\le ||[R^s,\zeta_0](\zeta_1u)||+||\zeta_0R^s(\zeta_1u)||+C||u||\\
&\le ||[R^s,\zeta_0](\zeta_1u)||+||R^s(\zeta_1u)||+C||u||. 
\end{split}
\end{equation}
We have, by pseudodifferential calculus,
\begin{equation}\Label{com1}
||[R^s,\zeta_0](\zeta_1u)|| \lesssim||R^{s-1}(\zeta_1u)||+||u||
\end{equation}
and, with $\zeta' \in C_c^{\infty}(U)$, $\zeta_1\prec \zeta'$,
\begin{equation}\Label{com2}
\begin{split}
||R^s(\zeta_1u)||=||R^s\zeta'\zeta_1u||
&\le||\zeta'R^s(\zeta_1u)||+||[R^s,\zeta'](\zeta_1u)||\\
&\le||\zeta'R^s(\zeta_1u)||+O(||R^{s-1}(\zeta_1u)||+||u||)\\
&\le||\zeta'R^s\zeta_1u||+C||u||,
\end{split}
\end{equation}
where the last inequality follows from the fact that $R^{s-1}-\zeta'R^{s-1}$ has order $-\infty$.
Combining \eqref{com1} and \eqref{com2} with \eqref{com0} we get
\begin{equation}\Label{combine}
|||\zeta_0u|||_s \lesssim ||\zeta'R^s(\zeta_1u)||+||u||.
\end{equation}
Since $R^s\zeta_1-R^s\zeta'$ is of order $-\infty$, we will replace from now on $\zeta'R^s\zeta_1$ by $\zeta'R^s\zeta'$, which we rename $\zeta_1 R^s \zeta_1$ (since the point here is that $\zeta_1\succ\sigma$).
By hypothesis \eqref{superlog} applied with $u$ replaced by $\zeta_1R^s(\zeta_1u)$ we have that, for every small $\delta>0$, there exists $C_\delta$ such that
\begin{equation}\Label{this}
||(\log\Lambda)\zeta_1R^s(\zeta_1u)|| \le \delta Q(\zeta_1R^s(\zeta_1u),\zeta_1R^s(\zeta_1u))+C_\delta||u||^2.  
\end{equation}
Recall that the condition $\Box u=\alpha$ is equivalent to $Q(u,v)=(\alpha,v)$ for all $(p,q)$-forms $v$.
So we have 

\begin{equation}\Label{bomba}
\begin{split}
Q(\zeta_1R^s(\zeta_1u),\zeta_1R^s(\zeta_1u))&=(\zeta_1R^s\zeta_1\bar\partial u, \bar\partial\zeta_1R^s(\zeta_1u))+
(\zeta_1R^s\zeta_1\bar\partial^* u, \bar\partial^*\zeta_1R^s(\zeta_1u))\\
&+([\bar\partial,\zeta_1R^s\zeta_1]u,\bar\partial\zeta_1R^s(\zeta_1u))+([\bar\partial^*,\zeta_1R^s\zeta_1]u,\bar\partial^*\zeta_1R^s(\zeta_1u) ) \\
&=(\bar\partial u,\zeta_1(R^{s})^*\zeta_1\bar\partial \zeta_1R^s(\zeta_1u))+(\bar\partial^* u,\zeta_1(R^{s})^*\zeta_1\bar\partial^* \zeta_1R^s(\zeta_1u )) + errI \\
&=Q(u,\zeta_1(R^s)^*(\zeta_1)^2R^s(\zeta_1 u))+(\bar\partial u, [\zeta_1(R^s)^*\zeta_1, \bar\partial]\zeta_1R^s(\zeta_1 u))\\
&+(\bar\partial^* u, [\zeta_1(R^s)^*\zeta_1, \bar\partial^*]\zeta_1R^s(\zeta_1 u)) + err_1\\
&=(\alpha,\zeta_1(R^s)^*(\zeta_1)^2R^s(\zeta_1 u)) + err_1 + err_{2}\\
&=(\zeta_1R^s(\zeta_1\alpha), \zeta_1R^s(\zeta_1 u))+ err_1 + err_{2}.
\end{split}
\end{equation}
If there were no errors, we would have finished.
In fact, it would be enough to use Cauchy-Schwarz in the last equality of \eqref{bomba}, give small and large constant $\epsilon$ and $\epsilon^{-1}$, and finally, after substituting in \eqref{this}, absorbe $\epsilon||\zeta_1R^s(\zeta_1 u)||$ to the left and combine with \eqref{combine} to get \eqref{result}.
But we have two error terms. Explicitly:
\begin{equation}\Label{errI}
err_1=([\bar\partial,\zeta_1R^s\zeta_1]u,\bar\partial\zeta_1R^s(\zeta_1u))+([\bar\partial^*,\zeta_1R^s\zeta_1]u,\bar\partial^*\zeta_1R^s(\zeta_1u) )
\end{equation}
and
\begin{equation}\Label{errII}
\begin{split}
err_{2}&=(\bar\partial u, [\zeta_1(R^s)^*\zeta_1, \bar\partial]\zeta_1R^s(\zeta_1 u))
+(\bar\partial^* u, [\zeta_1(R^s)^*\zeta_1, \bar\partial^*]\zeta_1R^s(\zeta_1 u))\\
&=([[\zeta_1(R^s)^*\zeta_1,\bar\partial)]^*,\bar\partial]u, \zeta_1R^s(\zeta_1 u)) + ([[\zeta_1(R^s)^*\zeta_1,\bar\partial^*)]^*,\bar\partial^*]u, \zeta_1R^s(\zeta_1 u))\\
&+([\zeta_1(R^s)^*\zeta_1,\bar\partial]^*u, \bar\partial^*\zeta_1R^s(\zeta_1 u))+
([\zeta_1(R^s)^*\zeta_1,\bar\partial^*]^*u, \bar\partial\zeta_1R^s(\zeta_1 u)).
\end{split}
\end{equation}
By the Jacobi identity we have
\begin{equation}\Label{Jacob}
[\bar\partial,\zeta_1R^s\zeta_1]=[\bar\partial, \zeta_1]R^s\zeta_1+\zeta_1[\bar\partial, R^s]\zeta_1+\zeta_1R^s[\bar\partial,\zeta_1].
\end{equation}
Since the support of the derivative of $\zeta_1$ is disjoint from the support of $\sigma$, the operator $[\bar\partial, \zeta_1]R^s\zeta_1+\zeta_1R^s[\bar\partial,\zeta_1]$ is bounded. But, in calculating the symbol $[\bar\partial, R^s]$, $\bar\partial$ hits $R^s$ and releases $\log\Lambda$. More precisely, we have that $||\zeta_1[\bar\partial, R^s]\zeta_1||\lesssim ||(\log\Lambda)\zeta_1R^s\zeta_1 u||+||u||$. With this observation on hand, we can control all the error terms. Using Cauchy-Schwarz and giving small and large constant, we get
\begin{equation}\Label{errorcontrolled2}
err_1 + err_{2} \le s.c.Q(\zeta_1 R^s(\zeta_1u),\zeta_1 R^s(\zeta_1u))+ l.c.(||(\log\Lambda)\zeta_1R^s\zeta_1 u||^2 + ||u||^2).
\end{equation}
We are almost done. Going back to \eqref{bomba} and bounding the errors with \eqref{errorcontrolled2}, we can absorbe $s.c.Q(\zeta_1 R^s(\zeta_1u),\zeta_1 R^s(\zeta_1u))$ on the left, so that the error now carries only the term  $l.c.(||(\log\Lambda)\zeta_1R^s\zeta_1 u||^2 + ||u||^2)$. We can forget about the $||u||^2$, which is present in the right side of the relation we want to obtain. But it is worth pointing out that even if $||(\log\Lambda)\zeta_1R^s\zeta_1 u||^2$ carries a large constant, when we substitute into \eqref{this} it gets multiplied by $\delta$, which is arbitrarily small and can be chosen in such a way that $||(\log\Lambda)\zeta_1R^s\zeta_1 u||^2$ comes in the end with a small constant. Hence this term too can be absorbed on the left in \eqref{this}. Combination with \eqref{combine} yelds \eqref{result} and hence \eqref{resultfake}. 
Now the result follows from elliptic regularization, as explained in \cite{CS01}. 
\end{proof}
\chapter{Superlogarithmic multipliers}
There are cases in which local regularity is not ruled by estimates but it is related to
the geometry of the domain. In \cite{K00} Kohn provided an example of a domain in which not
even superlogarithmic estimates hold, but we have local hypoellipticity for the $\bar\partial$-Neumann
problem. Kohn proved in fact this result for the tangential problem, but it was generalized to
the $\bar\partial$-Neumann by Baracco, Khanh and Zampieri in \cite{BKZ14}. Moreover, as we have already stated in the introduction,
Baracco, Pinton and Zampieri in  \cite{BPZ14} gave a fully geometrical explanation of the phenomenon, relating
local hypoellipticity to the presence of a sequence of cut-off $\{\eta\}$ such that the gradient $\di\eta$ and the Levi form $\di\dib\eta$ are are subelliptic multipliers.
In this section we improve the result of \cite{BPZ14} by
using superlogarithmic multipliers, which are weaker than subelliptic. In order to prove
our new result we need a twisted Morrey-Kohn-H\"ormander formula in which a general pseudodifferential operator
appears as a twisting term (twisted formulas are already known when the twisting term is
a function, cf. \cite{S10}). We need in fact to twist the formula not only by $\eta$ but also by $R^s$, where $R^s$ is the modification of the standard elliptic operator  of order $s$ introduced by Kohn in \cite{K02}.
\section{The  Kohn-H\"ormander-Morrey formula twisted by a pseudodifferential operator} 
We consider a general pseudodifferential operator $\Psi$. We work as usual on a pseudoconvex domain $\Omega$ with smooth boundary and within an adapted frame, where we denote as $\partial_j$ (resp $\bar\partial_j$) what we had previously indicated as $L_j$ (resp $\bar L_j$). We define constants
$c_{ij}=c_{ij}^n$ and $c_{ij}^h,\,\,i,j,h=1,...,n-1$ by means of the identity
\begin{equation}
\Label{2.5}
[\di_i,\dib_j]=c_{ij}(\di_n-\dib_n)+\sum_{h=1}^{n-1}c_{ij}^h\di_h-\sum_{h=1}^{n-1}\bar c_{ji}^h\dib_h.
\end{equation}
Thus $(c_{ij})$ is the matrix of the Levi form  $\di\dib r|_{T^\C b\Om}$ in the basis $\{\om_j\}$.
Moreover, we denote by Op$^\star$ an operator of order $\star$, and when $\star$ depends on $\Psi$ it is assumed that the support of the operator is contained in an arbitrarily small conical neighborhood of supp$\,\Psi$. In particular $\Opp$ refers to an operator of order 0 which depends on the $C^2$-norm of $b\Om$ but not on $\phi$ or $\Psi$. We also use the notation 
$$
Q^\phi_\Psi(u,u)=\NO{\Psi \dib u}_\phi+\NO{\Psi \dib^* u}_\phi,\quad u  \in   D_{\dib}\cap {D_{\dib^*}}.
$$
\bt
\Label{t2.1}
We have
\begin{equation}
\Label{2.1}
\begin{split}
\int_{b\Om}e^{-\phi}(c_{ij})(\Psi u,\overline{\Psi u})dV&+\int_\Om e^{-\phi}\di\dib\phi(\Psi u,\overline{\Psi u})dV\simleq Q^\phi_{\Psi}(u,\bar u)+\left|\int e^{-\phi}[\di,[\dib,\Psi^2]](u,\bar u)dV\right|
\\
&+\NO{[\di,\phi]\contrazione \Psi u}_\phi+\NO{[\di,\Psi]\contrazione u}_\phi+\Big|\sum_h\int (c_{ij}^h)\Big([\di_{h},\Psi](u),\bar u\Big)\,dV\Big|
\\
&+Q^\phi_{\Opm}(u,\bar u)+\NO{\Opm u}_\phi+\NO{\Psi u}_\phi.
\end{split}
\end{equation}
\et
\begin{remark}
In our application in the following section, $[\di,[\dib,\Psi^2]]$ and $[\di,\Psi]\contrazione$ have good estimates. Also, $\phi$ has ``selfbounded gradient", that is
$$
[\di,\phi]\contrazione<\di\dib \phi,
$$
where inequality is meant in the operator sense. In particular, the term in the right of \eqref{2.1} which involves $[\di,\phi]\contrazione$ is absorbed in the left.
\end{remark}
\begin{remark}
Formula \eqref{2.1} is also true for complex $\Psi$. In this case, one replaces $[\di,[\dib,\Psi^2]]$ by $[\di,[\dib,|\Psi|^2]]$ and add an additional error term $[\di,\bar\Psi]\contrazione$.
\end{remark}

\begin{proof}
We start from 
\begin{equation}
\Label{2.2}
e^{\phi}\Psi^{-2}[\dib_i,e^{-\phi}\Psi^2]=-\phi_{\bar i}+2\frac{[\dib_i,\Psi]}\Psi+\frac{\Op}{\Psi^2},
\end{equation}
whose sense is fully clear when both sides are multiplied by $\Psi^2$. We then have 
\begin{equation}
\Label{2.4}
\dib^*_{e^{-\phi}\Psi^2}=\dib^*+\di\phi\contrazione-2\frac{[\di,\Psi]}\Psi\contrazione+\frac{\Op}{\Psi^2},
\end{equation}
which is obtained through integration by parts and using \eqref{2.2} when $\bar\partial$ hits the weight $e^{-\phi}\Psi^2$. Similarly to what we did in Chapter 1, \eqref{2.5}, we are lead to define the following vector fields 
\begin{equation}
\Label{2.3}
\delta_i:=\di_i-\phi_i+2\frac{[\di_i,\Psi]}\Psi+\frac{\Op}{\Psi^2}+\Opp
\end{equation}
which express the components of $\dib^*_{e^{-\phi}\Psi^2}$.
Now, using the trivial identity $\di\dib=-\dib\di$, we have an analogous of \eqref{2.20}
\begin{equation}
\Label{2.6}
\begin{split}
[\delta_i&,\dib_j]=[\di_i,\dib_j]+\phi_{ij}-\sum_{h=1}^{n}c_{ij}^h\phi_h-2\frac{[\di_i,[\dib_j,\Psi]]}\Psi+2\frac{[\di_i,\Psi]\otimes[\dib_j,\Psi]}{\Psi^2}+\frac{\Op}{\Psi^2}+\Opp
\\
&=c_{ij}(\delta_n-\dib_n)-\sum_{h=1}^{n-1}\bar c_{ji}^h\dib_h+\sum_{h=1}^{n-1} c_{ij}^h\delta_h+\phi_{ij}-2\sum_hc^h_{ij}\frac{[\di_h,\Psi]}{\Psi}\contrazione
\\
&-2\frac{[\di_i,[\dib_j,\Psi]]}\Psi+2\frac{[\di_i,\Psi]\otimes[\dib_j,\Psi]}{\Psi^2}+\frac{\Op}{\Psi^2}+\Opp.
\end{split}
\end{equation}
We also have to observe that, being $\Psi$ a pseudodifferantial operator, in general $\Psi \dib^{(*)}u \Psi \dib^{(*)}u\ne \Psi^2 \dib^{(*)}u$. We then have (cf. \cite{BKPZ14})
$$
\NO{\Psi\dib^{(*)}u}_\phi=\int e^{-\phi}\Psi^2|\dib^{(*)}u|^2\,dV+Q_{\Opm}(u,\bar u)+\NO{\Opm u}_\phi+\No{\Opp\Psi u}_\phi.
$$
Following the same steps as in theorem ~\ref{kmh} we get the ``basic estimate with weight $e^{-\phi}\Psi^2$"
\begin{multline}
\Label{2.7}
\int_{b\Om}e^{-\phi}(c_{ij})(\Psi u,\overline{\Psi u})\,dV+\int_\Om[\di,[\dib,e^{-\phi}\Psi^2]](u,\bar u)\,dV
\\
-\NO{[\di,\phi]\contrazione \Psi u}_\phi-\NO{[\di,\Psi]\contrazione u}_\phi
+\sum_j\NO{\Psi\bar\di_j u}_\phi
\\
\simleq \NO{\Psi\dib u}_\phi+\NO{\Psi\dib^*_{e^{-\phi}\Psi^2}u}_\phi+sc\NO{\Psi\bar\nabla u}_\phi+\Big|\sum_h\int (c_{ij}^h)\Big([\di_{h},\Psi](u),\bar u\Big)\,dV\Big|
\\+Q^\phi_{\Opm}(u,\bar u)+\NO{\Opm u}_\phi+\NO{\Psi u}_\phi.
\end{multline}
In \eqref{2.7} we absorbe the term with small constant and rewrite $[\di,[\dib,e^{-\phi}\Psi^2]]$ by the aid of \eqref{2.2}; what we get is
\begin{multline}
\Label{2.8}
\int_{b\Om}e^{-\phi}\Psi^2 c_{ij}(u,\bar u)dV+\int_{\Om}e^{-\phi}\Psi^2 \phi_{ij}(u,\bar u)dV-\NO{[\di,\phi]\contrazione \Psi u}_\phi
\\+\int e^{-\phi}[\di_i,[\dib_j,\Psi^2]](u,\bar u)dV
-\NO{[\di,\Psi]\contrazione u}_\phi+\NO{\Psi\bar\nabla  u}_\phi
\\
\simleq \NO{\Psi\dib u}_\phi+\NO{\Psi\dib^*_{e^{-\phi}\Psi^2}u}_\phi
+Q^\phi_{\Opm}(u,\bar u)+\Big|\sum_h\int (c_{ij}^h)\Big([\di_{h},\Psi](u),\bar u\Big)\,dV\Big|
\\
+\NO{\Opm u}_\phi+\NO{\Psi u}_\phi.
\end{multline}
To carry out our proof we need to replace $\dib^*_{e^{-\phi}\Psi^2}$ by $\dib^*$. We have from \eqref{2.4}
\begin{equation}
\Label{2.9}
\begin{split}
\NO{\Psi \dib^*_{e^{-\phi}\Psi^2}u}_\phi&\simleq \NO{\Psi\dib u}_\phi+\NO{\Psi\di \phi\contrazione u}_\phi+\NO{[\dib,\Psi]\contrazione u}_\phi
\\
&+\underset{\T{\#}}{\underbrace{2\Big|\Re e (\Psi\dib^* u,\overline{\Psi\di\phi\contrazione u})_\phi\Big|+2\Big|\Re e (\Psi\dib^* u,\overline{[\di,\Psi]\contrazione u})_\phi+2\Big|\Re e (\Psi\di\phi\contrazione u,\overline{[\di,\Psi]\contrazione u})_\phi\Big|}}.
\end{split}
\end{equation}
We next estimate by Cauchy-Shwarz inequality
\begin{equation}
\Label{supernova}
\#\simleq \NO{\Psi\dib^* u}_\phi+\NO{\Psi\di\phi\contrazione u}_\phi+\NO{[\di,\Psi]\contrazione u}_\phi.
\end{equation}
We move the third, forth and fifth terms from the left to the right of \eqref{2.8}, use \eqref{2.9} and \eqref{supernova} and end up with \eqref{2.1}. 
\end{proof}
\section{\texorpdfstring{$F$}{F} type, twisted \texorpdfstring{$f$}{f} estimate and regularity of \texorpdfstring{$N$}{N}.}
\Label{s3}
We start by recalling a result by \cite{KZ10}. In our presentation it contains a specification of the estimate by the Levi form which is important in our application. We consider a bounded smoothly bounded pseudoconvex domain $\Om\subset\C^n$. For a form $u\in D_{\dib^*}$, let $u=\sum_k\Gamma_ku$ be the decomposition into wavelets (cf. \cite{K02}), and $Q(u,\bar u)=\NO{\dib u}+\NO{\dib^*u}$ the energy. We use the notation $(c_{ij})$  and $(\phi_{ij})$ for the Levi form of 
$b\Om$ and of a function $\phi$ respectively. We introduce a real function $F$ such that $\frac{F(s)}{s^2}\searrow0$ as $s\searrow0$ and set $f(t):=(F^*(t^{-\frac{1}{2}}))^{-1}$ where $F^*$ denotes the inverse.
\bt (cf. \cite{KZ10} Theorem~2.1)
\Label{t3.1}
Assume that $b\Om$ has type $F$ along a submanifold $S\subset b\Om$ of CR dimension $0$ in the sense that $(c_{ij})\simgeq \frac {F(d_S)}{d_S^2}\T{Id}$ where $d_S$ is the Euclidean distance to $S$ and Id the identity of $T^\C b\Om$. Then there is a uniformly bounded family of weights $\{\phi^k\}$ which yield the $f$ estimate
\begin{equation}
\Label{3.1}
\begin{split}
\NO{f(\Lambda)u}&\simleq \int_{b\Om}(c_{ij})\Lambda^{\frac{1}{2}}(u,\Lambda^{\frac{1}{2}}\bar u)\,dV+\sum_k\int_\Om (\phi^k_{ij})(\Gamma_ku,\overline{\Gamma_k u})+\NO{u}_0
\\
&\simleq Q(u,\bar u).
\end{split}\end{equation}
\et
\begin{proof}
We give two parallel proofs inspired to \cite{KZ10}, resp. \cite{BKPZ14}, which use the families of weights
$$
\psi^k:=-\log(\frac{-r}{2^{-k}}+1)+\chi(\frac{d_S}{a_k})\log(\frac{d_S^2}{a_k^2}+1)\quad\T{resp. }\phi^k:=\chi(\frac{d_S}{a_k})\log(\frac{d_S^2}{a_k^2}+1).
$$
Here $r=0$ is an equation for $b\Om$ with $r<0$ on $\Om$, $a_k:=F^*(2^{-k})$ and $\chi$ is a cut-off such that $\chi\equiv1$ in $[0,1]$ and $\chi\equiv0$ for $s\ge2$. 
We also use the notation $S_{a_k}$ for the strip $S_{a_k}:=\{z\in\Om:\,d_{b\Om}(z)<a_k\}$. 
Following word by word the proof of \cite{KZ10}, resp. \cite{BKPZ14}, we conclude
$$
\NO{f(\Lambda)u}_\Om\simleq\sum_k\int_{S_{2a_k}\setminus S_{a_k}}(c_{ij})(d_{b\Om}^{-\frac12}u,\overline{d_{b\Om}^{-\frac12}u})\,dV+\sum_k\int_\Om (\phi^k_{ij})(\Gamma_k u,\overline{\Gamma_ku})+\NO{u}_\Om
$$
resp.
$$
\NO{f(\Lambda)u}_\Om\simleq\int_{\Om}(c_{ij})(\Lambda^{\frac12}u,\overline{\Lambda^{\frac12}u})\,dV+\sum_k\int_\Om (\phi^k_{ij})(\Gamma_k u,\overline{\Gamma_ku})+\NO{u}_\Om.
$$
Finally, the conclusion follows from
$$
\no{d_{b\Om}^{-\frac12}u}_\Om\simleq \no{u}^b_0+\sum_{j=1}^n{\bar L_j u},\quad\T{ resp. }\no{\Lambda^{\frac12}u}_\Om\simleq \no{u}^b_0+\sum_{j=1}^n\no{\bar L_j u},
$$
according to \cite{K02} Section~8.
\end{proof}

We modify the weights $\phi^k$ to $\phi^k+t|z|^2$ so that their Levi form releases an additional  $t\T{Id}$ for $t$ big. They are absolutely uniformly bounded with respect to $k$ and to $t$ provided that we correspondingly shrink the neighborhood $U=U_t$.  Possibly by raising to exponential, boundedness implies ``selfboundedness of the gradient" when the weight is plurisubharmonic. In our case, in which to be positive is not $(\phi^k_{ij})$ itself but $2^k(c_{ij})+(\phi^k_{ij})$, we have, for $|z|$ small
\begin{equation}
\Label{stranova}
\begin{split}
|\di_b\phi|^2&=|\di_b(\phi^k+t|z|^2)|^2
\\
&\simleq |\di_b\phi^k|^2+t^2|z|^2
\\
&\leq 2^k(c_{ij})+(\phi^k_{ij})+t.
\end{split}
\end{equation}
Going back to \eqref{2.1} under this choice of $\phi$, we have that $\no{\di_b\phi\contrazione  \Psi u}^2$ can be removed from the right side. We combine Theorem~\ref{2.1} with Theorem~\ref{t3.1} formula \eqref{3.1}, observe that the weights $\phi^k$ can be removed from the norms by uniform boundedness,  and get the proof of the following

\bt
\Label{t3.2}
Let $b\Om$ have type $F$ along $S$ of CR dimension 0. Then we have the $f$ estimate
\begin{equation}
\Label{3.2}
\begin{split}
||&f(\Lambda)\Psi u||^2_0\simleq (c_{ij})(\Psi\Lambda^{\frac12}u,\overline{\Psi\Lambda^{\frac12}u})\,dV+\sum_k\int(\phi^k_{ij})(\Gamma_k\Psi u,\overline{\Gamma_k\Psi u})\,dV+\sum_j\NO{\bar L_j\Psi u}_0+t\NO{\Psi u}_0
\\
&\simleq Q_\Psi(u,\bar u)+\NO{[\di,\Psi]\contrazione u}_0+\int_\Om [\di,[\dib,\Psi^2]](u,\bar u)\,dV+Q^\phi_{\Opm}(u,\bar u)+\NO{\Opm u}_0+\NO{\Psi u}_0.
\end{split}
\end{equation}
\et
We have as application a criterion of regularity for the Neumann operator $N$ in a new class of domains.  Let $b\Om$ be ``block decomposed", that is, defined by $x_n=\sum_{j=1}^m h^{I^j}(z_{I^j},y_n)$ where $z=(z_{I^1},...,z_{I^m},z_n)$ is a decomposition of  coordinates.
\bt
\Label{t3.3}
Assume that
\begin{equation}
\Label{3.4}
\begin{cases}
 \T{ (a) $h^{I^j}$ has infraexponential type along a totally real $S^{I^j}\setminus\Gamma^{I^j}$ where $S^{I^j}$ is}
\\
\quad\T{  totally real  in $\C^{I^j}\times \C_{z_n}$ and $\Gamma^{I^j}$ is a curve of $\C^{I^j}\times\C_{z_n}$ transversal to $\C^{I^j}\times\{0\}$,}
\\
 \T{ (b) $h^j_{z_j}$ are superlogarithmic multipliers},
\\
\T{ (c) $c_{ij}^h$ are subelliptic multipliers.}
\end{cases}
\end{equation}
Then, we have local hypoellipticity of $\Box$ at $z_o=0$. 
\et
In the same class of domains, it is proved in \cite{BKPZ14} the hypoellipticity of the Kohn-Laplacian $\Box_b$. 

\begin{proof}
The proof is the same as in \cite{BKPZ14} Theorem~1.4. The argument is that, for a system  $\{\eta\}$ of cut-off, the vectors $\di\eta$  are superlogarithmic multipliers. To see it we start from
$$
\di\eta=(\di_b\eta,\di_\nu\eta),
$$
where $\di_\nu$ denotes the normal derivative.
Now, $\di_b\eta$ is a superlogarithmic multiplier by the hypothesis (i) and (ii). $\di_\nu\eta$ is $1$ but it hits $u_\nu$ which is 0 at $b\Om$ and therefore enjoys elliptic estimates. Therefore the full $\di\eta$ is a superlogarithmic multiplier.
\end{proof}

In case of a single block $x_n=h^{I^1}$ we regain \cite{BPZ14} which   transfers \cite{K00} from the tangential system to the $\dib$-Neumann problem and also gives a more general statement. The proof is far more efficient because it uses the elementary decomposition $\di\eta=(\di_b\eta,\di_\nu\eta)$ instead of $Q=Q^\tau\oplus \bar L_n$ (over tangential forms $u^\tau$) which requires the heavy technicalities of the harmonic extension.  
\br
\Label{r3.0,5}
The Levi form $(c_{ij})$ is a $\frac12$ subelliptic multiplier. If $b\Om$ is rigid and the Levi form is diagonal, then $(c_{ij}^h)\simleq (c_{ij})$ are also $\frac12$ subelliptic multipliers.
\er
\br
\Label{r3.1}
the above proof shows a general criterion. If $\di_b\eta$ is a superlogarithmic multiplier for $\Box_b$, then it is also for the $\dib$-Neumann problem; in this case therefore, the hypoellipticity of $\Box_b$ implies that of $\Box$.
\er

\noindent
{\it Example} 
Let $b\Om$ be defined by
$$
x_n=\sum_{j=1}^{n-1}e^{-\frac1{|z_j|^a}}e^{-\frac1{|x_j|^b}}\qquad\T{ for any $a\geq0$ and for $b<1$}.
$$
Then, \eqref{3.4} (a) is obtained starting from $h^j_{z_j\bar z_j}\simgeq \frac{e^{-\frac1{|x_j|^b}}}{|x_j|^2}$, that is, the condition of type $F^2_j:=e^{-\frac1{|\delta|^b}}$ along $S_j=\R_{x_j}\times\{0\}$. This yields the estimate of the $f$ norm for  $f(t)=
\log^{\frac1b} (t)$; since $\frac1b>1$, this is superlogarithmic. \eqref{3.4} (b) follows from $|h^j_{z_j}|^2\simleq h^j_{z_j\bar z_j}$ which says that the $h^j_{z_j}$'s are not only superlogarithmic, but indeed $\frac12$-subelliptic,  multipliers. Moreover, since $b\Om$ is rigid and the Levi form is diagona, then the $(c_{ij}^h)$ are subelliptic multipliers (cf. \ref{r3.0,5}).  Altogether we have that $\Box$ is hypoelliptic according to Theorem~\ref{t3.3}.

 \end{document}